\documentclass[10pt]{amsart}

\usepackage{amsmath, amsthm, amssymb}
\numberwithin{equation}{section}
\usepackage{xspace}
\usepackage{graphicx}
\usepackage{array}
\usepackage{braket}
\usepackage{multicol}
\usepackage{mathtools}
\usepackage{enumerate}
\usepackage{delarray}
\usepackage{mathtools}
\usepackage{faktor} % For quotients
\usepackage{mathrsfs}
\usepackage{tikz-cd}
\usepackage{hyperref}
\usepackage[normalem]{ulem} % For underlining
\usepackage[italicdiff]{physics} % For differentials
\usepackage{bbm} % For indicator
\usepackage{float}
\usepackage{stmaryrd} % for double brackets
\usepackage{aligned-overset}
\usepackage{xcolor}
\usepackage{cite}

\theoremstyle{plain}
\newtheorem{thm}{Theorem}[section]
\newtheorem{lemma}[thm]{Lemma}
\newtheorem{corollary}[thm]{Corollary}

\theoremstyle{remark}

\theoremstyle{definition}
\newtheorem{remark}[thm]{Remark}
\newtheorem{exmp}[thm]{Example}
\newtheorem{defn}[thm]{Definition}
\newtheorem*{notation*}{Notation}

\DeclarePairedDelimiterX{\inp}[2]{\langle}{\rangle}{#1, #2}

\makeatletter
\newcommand*{\bigcdot}{}% Check if undefined
\DeclareRobustCommand*{\bigcdot}{%
  \mathbin{\mathpalette\bigcdot@{}}%
}
\newcommand*{\bigcdot@scalefactor}{.5}
\newcommand*{\bigcdot@widthfactor}{1.15}
\newcommand*{\bigcdot@}[2]{%
  \sbox0{$#1\vcenter{}$}% math axis
  \sbox2{$#1\cdot\m@th$}%
  \hbox to \bigcdot@widthfactor\wd2{%
    \hfil
    \raise\ht0\hbox{%
      \scalebox{\bigcdot@scalefactor}{%
        \lower\ht0\hbox{$#1\bullet\m@th$}%
      }%
    }%
    \hfil
  }%
}
\makeatother

\newcommand{\C}{\mathbb{C}}

\newcommand{\B}{\mathbb{B}\xspace}
\newcommand{\M}{\mathbb{M}\xspace}

\renewcommand{\S}{\mathbb{S}\xspace}

\renewcommand{\H}{\mathcal{H}}
\newcommand{\K}{\mathcal{K}}

\DeclareMathOperator{\nor}{nor}

\usepackage{enumitem}
\usepackage{color}
\newlist{steps}{enumerate}{1}
\usepackage[foot]{amsaddr}
\setlist[steps, 1]{label = Step \arabic*:}
\hypersetup{%
  colorlinks=true,%
  linkcolor=blue,%
  citecolor=blue,%
  filecolor=blue,%
  menucolor=blue,%
  urlcolor=blue,%
  pdfnewwindow=true,%
  pdfstartview=FitBH
}   
\title{pp}

\newcommand{\KK}{{\mathcal K}}

\newcommand{\KKK}{{\mathbb K}}

\newcommand{\BB}{{\mathbb{ B}}}

\newcommand{\XX}{{\mathbb{ X}}}
\newcommand{\HH}{{\mathcal H}}
\newcommand{\VV}{{\mathcal V}}

\newcommand{\EE}{{\mathcal E}}

\newcommand{\GG}{{\mathcal G}}

\begin{document}

\title[Relative biexactness and mixing in von Neumann algebras]{Relative biexactness and mixing in von Neumann algebras}

\author[S. Kunnawalkam Elayavalli]{Srivatsav Kunnawalkam Elayavalli}
\address{Department of Mathematics, UC Berkeley, 970 Evans Hall, Berkeley, CA 94720
}\email{sriva@umd.edu}
%\urladdr{https://sites.google.com/view/srivatsavke/home}

\author[Z. Yang]{Zhiyuan Yang}
\address{Department of Mathematical, Texas A\&M, 155 Ireland St, College Station, TX 77843, USA}\email{zhiyuanyang@tamu.edu}%\urladdr{https://zhiyuanyang7.github.io/}

\begin{abstract}
    We develop a new technique to upgrade relative biexactness in general von Neumann algebras: suppose that $\{N_i\}_{i\in I}\subset M$ are mixing and biexact subalgebras of a separable von Neumann algebra with expectation, and if $M$ is biexact relative to  $\{N_i\}_{i\in I}$, then $M$ is biexact. This result yields several new examples of biexact von Neumann algebras, notably including amalgamated free products. By generalizing the relative biexactness results of Hoshino to the von Neumann algebra setting and applying our result above along with certain bimodule computations, we in fact obtain, as an application, a new classification result for biexactness for graph products of finite dimensional von Neumann algebras. This yields significant extensions of prior works of Caspers-Borst, and Blufstein-Goldman-Oyakawa. 
\end{abstract}

\maketitle

\section{Introduction}

The study of \emph{biexact} groups was initiated by Ozawa in his article \cite{Oza04}. This article not only settled the problem of Ge, proving solidity for $L(\mathbb{F}_n)$, but also opened up line of research that settled multiple long standing open problems \cite{OzPo10a, CS13, PopaVaesFree, BIP21}. Biexactness demands that the left action of the group on the \emph{small at infinity boundary} (the fixed points of the Stone-Cech remainder under the right action) is topologically amenable. Multiple groups are known to be biexact, including families of hyperbolic groups, relatively hyperbolic groups, wreath products, small cancelation groups, etc. On the other hand, $G$ is not biexact whenever it contains a copy of $K_1\times K_2$ where $K_1$ is infinite and $K_2$ is nonamenable. More generally, $L(G)$ is solid, i.e, the commutant of any diffuse subalgebra is amenable. Combined with the weak amenability \cite{CowlingHaagerup1989}, one in fact has strong solidity of $L(G)$ \cite{CS13}, which demands that the normalizer of any diffuse amenable subalgebra is amenable. 

Generalizing this notion to von Neumann algebras is an apriori hard problem because of the fundamental difficulties stemming from works of Johnson-Parrot \cite{johnson1972operators} and Popa \cite{popa1987commutant} in obtaining a satisfactory analogue of the small at infinity boundary. This was circumvented in the article \cite{DKEP23}, wherein a natural von Neumann algebra boundary theory was developed with various applications. They introduced a fruitful new topological perspective to study von Neumann algebras building on the theory of operator $M$-bimodules developed in works of Magajna \cite{Ma97, Magajna00}. By further developing a theory of nuclear maps in von Neumann algebras, \cite{ding2023biexact} established the notion of a biexact von Neumann algebra, and unified certain approaches towards strong solidity under this framework. Various other results pertaining to this line of research have appeared recently \cite{DKE24properproximalgroups, DKE24, ding2024first, Toyosawa2025Weak, ding2025unique, toyosawa2025relative, ding2025relative, isono2025weak, ding2025structural}. In particular, the work \cite{DKE24} developed a technique to bootstrap proper proximality in von Neumann algebras, and used it to derive applications to the structure of free products (see also \cite{ding2024first}, for another application to cocycle superrigidity). 

In this paper, we develop a new method to upgrade relative biexactness in general von Neumann algebras, with various applications. For the basic aspects of the underlying theory of boundary pieces used  in this paper, we direct the reader to Section \ref{prelims}. The main result of our paper is the following:

\begin{thm}\label{main intro}
        Let $M$ be a separable von Neumann algebra and $\mathcal{N}=\{N_i\}_{i\in I}$ be a family of subalgebras with expectation and with commuting boundary pieces. If $M$ is biexact relative to $\mathcal{N}$, and each $N_i$ is mixing and biexact, then $M$ is biexact.
\end{thm}

We first clarify the precise technical meaning of ``commuting boundary pieces'' can be seen in Theorem \ref{thm: main}; it asks for the support projections of the boundary pieces in the relevant bidual space to commute. This condition may appear rather unmotivated to the naked eye, but it holds in most cases in practice. Mixingness is a standard notion in the bimodule theory of von Neumann algebras, motivated from the analogous notion in ergodic theory. For the definition we direct the reader to \cite{AP}. As a first concrete application of our result, we produce a new source of biexact von Neumann algebras arising from amalgamated free products. This result follows from Theorem \ref{main intro} combined with the relative biexactness result of \cite{toyosawa2025relative}. In particular this offers a significant generalization of Corollary 1.7 of \cite{toyosawa2025relative}.

\begin{corollary} 

        Let $M_1,M_2$ be two separable biexact tracial von Neumann algebras with a common subalgebra $B$. If $B$ is amenable and mixing in both $M_1$ and $M_2$, then the amalgamated free product $ M=M_1\ast_BM_2 $ is biexact.
        
\end{corollary}

As a simple case, our main Theorem \ref{main intro} yields examples even at the group level. Let $G$ be a countable group. Recall that a subgroup $H<G$ is \emph{almost malnormal} if $gHg^{-1}\cap H$ is finite for every $g\in G\setminus H$. It is well known that $H<G$ is almost malnormal implies that $L(H)\subset L(G)$ is mixing (see for instance \cite{BoutonnetCarderi2017}). Now let $G$ be a countable group that is biexact relative to a family of subgroups $ \{H_i\}_{i\in I} $. Theorem \ref{main intro} then implies that if each $H_i$ is biexact, and almost malnormal in $G$, then $G$ is biexact. In Section \ref{relatively biexact section end} we give two independent and more accessible proofs of this (this point will be revisited at the end of the introduction). This result allows us to access biexactness for a concretely new family of groups combining with the result of Oyakawa \cite{oyakawa2024infinite}.

\begin{corollary}\label{Koichi}
    Suppose that $\Gamma$ is a uniformly fine hyperbolic countable graph with $\operatorname{girth}(\Gamma) > 20$ and that $\mathcal{G} = \{G_v\}_{v \in V(\Gamma)}$ is a collection of non-trivial finite groups. Then, the graph product $\Gamma\mathcal{G}$ is biexact.
\end{corollary}

As the above result suggests, a natural setting where our main result can apply naturally is to the setting of graph products of von Neumann algebras \cite{CaFi17}, which has seen immense activity in recent years (see \cite{horbez2025rigiditygraphproductvon} and the references therein). In fact, our next main results will give vast generalizations of the above Corollary \ref{Koichi}.  We first prove the correct relative biexactness result here, generalizing the prior work of Hoshino \cite{hoshino2026relative}. In particular, we prove biexactness relative to the family of links in arbitrary graph products of von Neumann algebras.

\begin{thm}
         If $\GG$ is a graph and $M_v$ is biexact for each $v\in \mathcal{V}(\GG)$, then the graph product von Neumann algebra $ M=\ast_{v\in \GG}(M,\varphi_v)$ is biexact relative to the family $ \{ M_{\text{link}_v} \}_{v\in \VV}$, where $  M_{\text{link}_v} = \ast_{v\in \text{link}_v} (M_v,\varphi_v) $.
\end{thm}

Combining the above result and the bimodule calculation from \cite{charlesworth2025structure} with our main Theorem \ref{thm: main}, we are able to obtain a full classification of biexactness for graph products of finite dimensional von Neumann algebras.

\begin{corollary}\label{graph prod classification corollary}
       Let $\GG$ be a simple graph without an infinite clique. The followings are equivalent:
    \begin{enumerate}
        \item For any family of finite dimensional von Neumann algebras each with dimension greater than 2 and with faithful normal states $\{(M_v,\varphi_v)\}_{v\in \VV(\GG)}$, the graph product von Neumann $M=\ast_{v\in \GG}(M_v,\varphi_v)$ is biexact.
        \item There is no square in $\GG$: there does not exist distinct vertices $v_1,v_2,u_1,u_2\in \VV(\GG)$ such that $ v_1\nsim v_2$, $u_1\nsim u_2$, and $ v_i\sim u_j $ for each $i,j=1,2$.
    \end{enumerate}
\end{corollary}

 As a quick remark, note that the condition in Corollary \ref{graph prod classification corollary} that each vertex algebra has dimension greater than $2$ is added to facilitate a clean statement and avoid the nuisance arising from the fact that free products of two dimensional von Neumann algebras are amenable. Corollary \ref{graph prod classification corollary} extends the strong solidity results obtained in \cite{BorstCaspers2024}. Indeed, arbitrary graph products of finite dimensional von Neumann algebras possess the complete metric approximation property \cite{borst2024ccap}, and together with biexactness this yields strong solidity. Note that the result of \cite{blufstein2025strongsolidityclassificationcoxeter} proves biexactness for right angled Coxeter groups whose von Neumann algebras are strongly solid. Therefore our results can also be viewed as an extension of \cite{blufstein2025strongsolidityclassificationcoxeter} in a different direction.

We now discuss some of the insights going into the proof of Theorem \ref{thm: main}. A systematic approach of relating the small at infinity boundary of a von Neumann algebra to its relative small at infinity boundary with respect to mixing subalgebras was initiated by \cite{DKE24}, under technical assumptions regarding the Pimsner-Popa basis (see also \cite{ding2024first} which handles the group case). In \cite{toyosawa2025relative}, this technique is generalized to arbitrary mixing subalgebras $N \subset M$ with expecation. An essential idea in \cite{DKE24,toyosawa2025relative} is that if we consider the small at infinity boundary $\tilde{\S}(M)$ inside the bidual (see Section 2.2 for the precise definition), then we have access to the necessary projections to cut the embedding $M \hookrightarrow\tilde{\S}(M)$ into three pieces. The key to this decomposition is that two of these pieces belong entirely to the relative small at infinity boundary. To deal with the remaining piece ($ \tilde{\S}(M)q_{\XX_{N}}q_{\KKK}^{\perp}$ in Corollary \ref{cor: biexact mixing subalg}), we need the observation that there is another copy of the basic construction $\langle M, e_N \rangle$ inside. Using this observation, it is shown in \cite{toyosawa2025relative} that if $M$ is biexact relative to a mixing amenable subalgebra $N$, then $M$ is biexact. This proof heavily relies on the amenability of $N$ to guarantee the weak nuclearity of the embedding $M\subset \langle M,e_N\rangle$. This is exactly where the difficulty lies in generalizing the argument (to bypass amenability). We are able to overcome this difficulty with a new insight of extending the related map $ \langle M, e_N \rangle \to \tilde{\S}(M)$ to the whole $\B(L^2M)$ when $N$ is biexact. This is achieved in Sections 3.1 and 3.2 by carefully applying the biexactness of $N$ and dealing with subtle aspects of the related $N$ and $N'$ bimodule topology in $\B(L^2M)$. This then allows us to account for all the pieces and follow the upgrading procedure to get the desired result. 

We conclude the introduction by discussing another result, suggested to us by J. Peterson, that characterizes relative biexactness in the group setting.  We include it here because it is of independent interest and conceptually bridges our general von Neumann algebra approach with the group setting. Specifically, while there is already a shorter proof of the group theoretic analogue of Theorem \ref{main intro} (see Corollary \ref{cor: group case main thm}), this characterization provides a unified approach to adapt the techniques from our main theorem into an alternative, purely group-theoretic proof for countable $\Gamma$.

\begin{thm}
    Let $\Gamma$ be a countable discrete exact group, and $I\subset \ell^\infty(\Gamma)$ be a $\Gamma$-boundary piece. The followings are equivalent
    \begin{enumerate}
        \item $\Gamma$ is biexact relative to $I$.
        \item For every finite subset $ E\subset \Gamma $ and $\varepsilon>0$, there exists $\mu: \Gamma\to \text{Prob}(\Gamma)$ such that for all $s,t\in \Gamma$, the subset
        \[ \{x\in \Gamma: \|\mu(sxt)-s\cdot\mu(x)\|_1\geq \varepsilon\} \] is small relative to $I$. 
        \item The left action of $\Gamma$ on the abelian von Neumann algebra
    \[ \left [\left(\ell^\infty( \Gamma )/I\right)^{**}\right]^{\Gamma_r} \] is amenable in the sense of Zimmer.
    \item For every finite subset $E\subset \Gamma$ and every $\delta>0$, there exists a finite subset $K\subset \Gamma$ and elements $a_g\in (\ell^\infty(\Gamma)/I)_+$ with $g\in K$, such that $ \sum_{g\in K}a_g=1 $ and, putting $a_g=0$ outside $K$, for all $s\in E$,
    \[ \|\sum_{g\in \Gamma}|s\cdot a_g-a_{sg}|\|<\delta,\quad \|   \sum_{g\in \Gamma}| \rho_s(a_g)-a_{g}|\|<\delta\]
    \end{enumerate}
\end{thm}

\subsection*{Acknowledgements} This work was done when the second author visited the first author at the University of Maryland, College Park in March 2026. We thank UMD for the stimulating environment. We warmly thank J. Peterson, D. Jekel, K. Toyosawa, and D. Gao for helpful feedback.   

\subsection*{AI statement} No AI tools were used by the authors in any stage of the research process in this article.

\section{Preliminaries}\label{prelims}

We begin with a brief introduction to small-at-infinity boundary and biexactness of von Neumann algebras. We refer the reader to \cite{DKEP23},\cite{ding2023biexact} for more details.

\subsection{Biexact and properly proximal von Neumann algebras}
\subsubsection{Strong and weak $M$-topology on operator bimodules}
Let $\HH$ be a complex Hilbert space, and $M,N$ two von Neumann subalgebras of $\B(\HH)$. We say an operator system $X\subset \B(\HH)$ is a $M$-$N$ operator bimodule if $ MXN\subset X $.
Given a normal state $\omega\in M_*$ and a normal state $\rho\in M'_*$, we can define a seminorm on $X$ as in \cite{Magajna00}.
\[ s_{\omega,\rho}(T):= \inf\{\omega(a^{*}a)^{\frac{1}{2}}\|S\|\rho(b^{*}b)^{\frac{1}{2}} : x= a^{*}Sb, a\in M,b\in N, y\in X\} \]
\begin{defn}
    Given a $ M $-$N$ operator bimodule $X$,  the \textbf{$M$-$N$ topology} on $X$ is the locally convex topology given by the family of seminorms $\{s_{\omega,\rho}\}$ with normal state $\omega\in M_*$ and normal state $\rho\in N_*$.
\end{defn}

We are specifically interested in the following case: If $M$ is a von Neumann subalgebra of $\B(\HH)$, and $M'$ be the commutant of $M$ inside $\B(\HH)$. Then $\B(\HH)$ can be considered as a $M$-$M$-bimodule as well as a $M'$-$M'$-bimodule (also as $M$-$M'$ and $M'$-$M$ bimodule).

We want to consider the finest locally convex topology on $\B(\HH)$ contained in all those four bimodule topologies. This topology can again be given via seminorms: Given a normal state $\omega\in M_*$ and a normal state $\rho\in M'_*$, we define the seminorm on $ \B(\H)$:
\begin{multline*} r_{\omega,\rho}(T):=  \inf\{ (\rho(a^*a)+\omega(b^*b))^{1/2}\|Z\|(\rho(c^*c)+\omega(d^*d))^{1/2} \\ \;\big|\; T= \begin{pmatrix}
    a\\
    b
\end{pmatrix}^* Z  \begin{pmatrix}
    c\\
    d
\end{pmatrix}, a,c\in M',b,d\in M,Z\in \mathbb{M}_2(\BB(\HH)\}, 
\end{multline*}
\begin{defn}
    The \textbf{$M$-$M$ and $M'$-$M'$ topology} on $\B(\HH)$ is the locally convex topology given by the family of seminorms $\{r_{\omega,\rho}\}$ with normal state $\omega\in M_*$ and normal state $\rho\in M'_*$.
\end{defn}

The dual space of the $M$-$N$ topology on an operator bimodule $X$ is given as follows.
\begin{thm}[{\cite[Theorem 3.7]{Magajna00}} ]\label{Thm: weak topology}
    A linear functional $ f\in X^*$ is continuous with respect to the $M$-$N$ topology if and only if for each $ T\in X$, the mapping $ M\times N\ni (a,b)\mapsto f(aTb) $  is normal on each variables.
\end{thm}
Therefore, we denote $X^{M\sharp N}$ to be the subspace of $X^*$ consisting of linear functionals continuous with respect to the $M$-$N$ topology.

\begin{defn}
    The \textbf{weak $M$-$N$ topology} on a $M$-$N$ operator bimodule $X$ is the locally convex topology given by the linear functionals in $ X^{M\sharp N}$.
\end{defn}

\begin{remark}
    When $M=N$ and $X$ is a $C^*$-algebra containing $M$, then $ X^{M\sharp M} $ is closed under Jordan decomposition \cite[Lemma 5.3]{DKEP23}. Also, by the Cauchy-Schwarz inequality, a state $\varphi$ belongs to $X^{M\sharp M}$ if and only if $\varphi$ is normal on $M$. In particular $ X^{M\sharp M} $ is the span of states that are normal on $M$.
\end{remark}

Similarly, a linear functional $f$ on $\B(\H)$ is continuous with respect to the $M$-$M$ and $M'$-$M'$ topology if and only if $ f\in \B(\H)^{M\sharp M}\cap \B(\H)^{M'\sharp M'} $. When $\H = L^2M$, we denote in short $ \B(L^2M)^{\sharp}:=\B(L^2M)^{M\sharp M} $ and $ \B(L^2M)^{\sharp}_J:=\B(L^2M)^{M\sharp M}\cap\B(\H)^{JMJ\sharp JMJ} $ following \cite{DKEP23}. In particular, $ \B(L^2M)^{\sharp}_J $ is the span of states that are normal on both $M$ and $M'$.

\subsubsection{Boundary pieces of von Neumann algebras}
For a von Neumann algebra $M\subset \B(L^2M)$, a $M$-boundary piece is a hereditary $C^*$-subalgebra $\XX \subset \B(L^2M)$ such that the multiplier $ \mathcal{M}(\XX) $ contains an ultraweak dense subalgebra of $M$ as well as of $JMJ$. We are going to focus on the subalgebra boundary piece: Let $ \mathcal{N} = \{N_i\}_{i\in I} $ be a family of von Neumann subalgebras of $ M$ with expectation and with the Jones projections $e_{N_i}$'s. We consider the smallest $M$-boundary piece containing $ \{e_{N_{i}}\}_{i\in I} $:
$$ \XX_{ \mathcal{N} } := \overline{ MJMJ \left( C^*( e_{N_i}\B(L^2M) e_{N_j}:i,j\in I  )\right) JMJM }. $$
We also denote $ \XX_{ N }= \XX_{\{N\}} = \overline{ MJMJ e_N\B(L^2N)e_N MJMJ } $. And when $ N=\C $, we precisely obtain the compact operators $ \XX_{\C}=\K(L^2M) $. Note that for a subalgebra boundary piece $ \XX_{\mathcal{N}} $, its multiplier always contains the whole $M$ and $JMJ$. For a general boundary piece $\XX$, one can always replace it by a slightly larger boundary piece $ \KKK_\XX(M) $ which contains $M$ and $JMJ$ in its multiplier without changing the $ M$-$M$ and $JMJ$-$JMJ$ closure (see \cite[Section 3]{DKEP23}). In this paper, for simplicity, we always assume that a boundary piece $\XX$ contains $M$ and $JMJ$ in its multiplier. 

For each $M$-boundary piece $\XX$, we denote $ \KKK^{\infty,1}_\XX( M ) $ its closure under the $M$-$M$ and $JMJ$-$JMJ$ topology (when $ M $ is tracial, $\KKK^{\infty,1}_\XX( M )$ is precisely the closure of $\XX$ under the $L^\infty$-$L^1$ norm by \cite[Prop. 3.1]{DKEP23}). 

\subsubsection{Small-at-infinity boundary}
Given a $M$-boundary piece $\XX$, its small-at-infinity boundary is defined in \cite{DKEP23} as the operator system
$$ \S_{\XX}(M):=\{T\in \B(L^2M): [T,JxJ]\in \KKK^{\infty,1}_{\XX}(M),\forall x\in M\}. $$
And when $ \XX=\K(L^2M)$, one simply denotes $ \S(M):=\S_{\K(L^2M)}(M) $.

\begin{defn}
    A von Neumann algebra $M$ is said to be \textbf{biexact relative} to a boundary piece $\XX$ if the embedding
    $M\subset \S_\XX(M)$ is $M$-nuclear, i.e. there exists a net of contractive completely positive (c.c.p.) maps $\phi_n: M\to \S_\XX(M)$ factoring through matrix algebras such that $\phi_n$ pointwisely converges to the identity map on $M$ with respect to the $M$-$M$-topology.
\end{defn}

This definition of relative biexactness coincides with the relative biexactness of groups. See Section \ref{relatively biexact section end} for details about the biexactness for groups.
\begin{thm}[{\cite[Theorem 6.2]{ding2023biexact}}]\label{thm: DP 6.2}
    Let $\Gamma$ be a discrete group with a family of subgroups $ \{ H_i \}_{i\in I} $, then $ \Gamma $ is biexact relative to $ \{H_i\}_{i\in I} $ if and only if the group von Neumann algebra $ L\Gamma $ is biexact relative to $ \{ LH_i\}_{i\in I} $. 
\end{thm}

\subsection{Boundary piece and biexactness through the bidual}

Fix a von Neumann algebra $M\subset \B(L^2M)$, we consider the bidual von Neumann algebra $\B(L^2M)^{**}$ equipped with the canonical embedding $\iota: \B(L^2M)\to \B(L^2M)^{**} $. For any von Neumann subalgebra $P\subset \B(L^2M)$, we denote by $p_{\text{nor}}^P\in P^{**}\subset \B(L^2M)^{**}$ the support central projection of the canonical $*$-homomorphism $P^{**}\to P$ (i.e. $p_{\text{nor}}^P$ is the canonical projection such that $ P^{**}p_{\text{nor}}^P \simeq P $). We denote in short $p_{\nor} := p_{\nor}^M \subset \B(L^2M)^{**}$ and $p_{\nor}^\sharp = p_{\nor}^Mp^{JMJ}_{\nor}\in \B(L^2M)^{**}$.

If $\B(L^2M)^{**}$ has a normal faithful representation on a Hilbert space $\K$, then $ p_{\text{nor}}^P\K $ is the space of vectors $\xi$ such that the vector state $ \varphi_\xi:=\langle \xi,\cdot \xi \rangle$ is normal on $P$. Therefore, we could identify $ \B(L^2M)^{\sharp }_J $ as the predual of the von Neumann algebra $ p_{\nor}^\sharp \B(L^2M)^{**}p_{\nor}^\sharp $. $\B(L^2M)$ also canonically embeds into $ p_{\nor}^\sharp \B(L^2M)^{**}p_{\nor}^\sharp= \B(L^2M)_J^{\sharp *} $ via the map
$$ \iota^{\sharp}:= \text{Ad}_{p^\sharp_{\nor}}\circ \iota : \B(L^2M) \to \B(L^2M)^{\sharp*}_J.$$
Note that $\iota^\sharp$ is only a completely order isomorphism but not a homomophism in general, as $p^\sharp_{\nor}$ does not commute with all of $\B(L^2M)^{**}$. Also, as both $p^{M}_{\nor}$ and $ p^{JMJ}_{\nor} $ commute with $ M^{**} \vee (JMJ)^{**}$, both $M$ and $JMJ$ are in the multiplicative domain of $\iota^\sharp$. 

Similarly, for a $M$-boundary piece $\XX\subset \B(L^2M)$, we can also identify $ \XX^{\sharp *}_J = p^\sharp_{\nor}\XX^{**}p^\sharp_{\nor}$, where $ \XX^{\sharp }_J := \XX^{M\sharp M}\cap \XX^{JMJ\sharp JMJ}\subset \XX^*$. Since $\XX$ is hereditary, for any approximate identity $e_{j}$, $ \lim_j \iota(e_j)$ is the identity of the von Neumann algebra $ \XX^{**} $. We will denote this identity by $ \bar{q}_{\XX} $. Note that by the Kaplansky density theorem, $ \bar{q}_{\XX} \B(L^2M)^{**}\bar{q}_{\XX} = \XX^{**} $ (i.e., the double dual of a hereditary subalgebra is always a corner subalgebra).

Since $M$ and $JMJ$ are inside the multiplier of $\XX$, both $\iota(M)$ and $\iota(JMJ)$ commute with $\bar{q}_{\XX}$, and in particular, $ p^\sharp_{\nor}\in M^{**}\vee (JMJ)^{**} $ commutes with $ \bar{q}_{\XX} $. Denote $ q_{\XX}:= p^\sharp_{\nor} \bar{q}_{\XX} $, then $ q_{\XX} $ is the identity of the von Neumann algebra $ \XX^{\sharp*}_J = q_{\XX} \B(L^2M)^{\sharp *}_J q_{\XX}$.

Recall that $\KKK^{\infty,1}_{\XX}(M)$ is the $M$-$M$ and $JMJ$-$JMJ$ closure of $\XX$, by Theorem \ref{Thm: weak topology}, it is also the weak $M$-$M$ and $JMJ$-$JMJ$ closure of $\XX$. In particular, we have $$ \KKK^{\infty,1}_{\XX}(M)= (\iota^{\sharp})^{-1}(\XX^{\sharp *}_J). $$

Following \cite[Lemma 8.5]{DKEP23}, we define the bidual version of the small-at-infinity boundary $\S_{\XX}(M)$ as:
$$\tilde{\S}_{\XX}(M) := \{T\in \B(L^2M)^{\sharp*}_J: [T,\iota^{\sharp}(JxJ)]\in \XX^{\sharp *}_J,\forall x\in M\} \subset \B(L^2M)^{\sharp*}_J.$$
\begin{remark}
    When $\XX=\K(L^2M)$, $\tilde{\S}(M):=\tilde{\S}_{\K(L^2M)}(M)$ is in fact a von Neumann algebra. Indeed, since $\K(L^2M)$ is an ideal, the identity $q_{\KKK}$ of $ \K(L^2M)^{\sharp *}_J $ is a central projection in $ \B(L^2M)^{\sharp*}_J $. In particular, we have
    $$ \tilde{\S}(M) = \iota(JMJ)'q_{\KKK}^{\perp}\oplus \B(L^2M)^{\sharp*}_Jq_{\KKK} = \iota(JMJ)'q_{\KKK}^{\perp}\oplus \K(L^2M)^{\sharp *}_J. $$
\end{remark}

\begin{remark}
    For a family $ \mathcal{N} = \{N_i\}_{i\in I} $ of subalgebras with expectation, we have (\cite[Lemma 5.3]{toyosawa2025relative})
    $$ q_{ \XX_\mathcal{N} } = \vee_{i\in I}\vee_{ u,v\in \mathcal{U}(M) }\iota^{\sharp}( uJvJ e_{N_i}Jv^*Ju^* ) = \vee_{i\in I}q_{\XX_{N_i}}. $$
\end{remark}

We can also characterize biexactness of a separable $M$ using $ \tilde{\S}_{\XX}(M) $.
\begin{thm}[{\cite[Theorem 5.21]{toyosawa2025relative}}]\label{thm: bidual characterization of biexact}
    A separable von Neumann algebra $M$ is biexact relative to $\XX$ if and only if the embedding $\iota^{\sharp}|_{M}: M\to \tilde{\S}_{\XX}(M)$ is weak${}^*$-nuclear.
\end{thm}

\subsubsection{Biduals of some other non-unital subalgebras}
Given a subalgebra $N\subset M$ with expectation. Besides the boundary piece $ \XX_N $, we can also consider other non-unital subalgebras inside $\B(L^2M)$:
 \[ \overline{Me_NM},\quad \overline{JMJe_NJMJ}, \]
 as well as their hereditary version
 \[\overline{Me_N \B(L^2M)e_NM},\quad \overline{JMJe_N \B(L^2M)e_NJMJ}.\]
 The bidual subalgebras $ (Me_NM)^{**} $ and $ (Me_N \B(L^2M)e_NM)^{**} $ share the same identity $ \bar{q}_{Me_NM} := \vee_{u\in \mathcal{U}(M)}\iota(ue_Nu^*)\in \B(L^2M)^{**} $. Since $\overline{Me_N \B(L^2M)e_NM}$ is hereditary, again by Kaplansky, we also have $$ \bar{q}_{Me_NM}\B(L^2M)^{**}\bar{q}_{Me_NM} = (Me_N \B(L^2M)e_NM)^{**}. $$ The same claims also hold for $ \overline{JMJe_NJMJ} $ and $ \overline{JMJe_N \B(L^2M)e_NJMJ} $.

 Note that for any subalgebra $N\subset M$ with expectation, $ \iota(e_{N}) $ always commutes with $ p^M_{\nor} $ and $p^{JMJ}_{\nor}$ (\cite[Lemma 3.8]{DKE24}). Therefore, we also have $ [ \bar{q}_{Me_NM},p^{\sharp}_{\nor} ]= [ \bar{q}_{JMJe_NJMJ},p^{\sharp}_{\nor} ] =0 $. We will denote $ q_{Me_BM}:=\bar{q}_{Me_NM} \;p^{\sharp}_{\nor} $ and $ {q}_{JMJe_NJMJ}:=\bar{q}_{JMJe_NJMJ}\;p^{\sharp}_{\nor} $. See \cite[Section 5.1]{toyosawa2025relative} for more details of those projections.

\subsection{Mixing subalgebras without a trace}
We follow \cite[Section 5.2]{toyosawa2025relative} and define mixing subalgebras of general von Neumann algebras as follows.
\begin{defn}
    Let $N\subset M$ be a subalgebra with expectation. We say $N$ is a \textbf{mixing subalgebra} if for every $x,y\in M\ominus N$, $ e_NxJyJe_N \in \KKK^{\infty,1}(N) $.
\end{defn}

We say that a subalgebra $N\subset M$ with expectation is \emph{coarse} if $L^2(M\ominus N)$ as an $N$--$N$ bimodule is contained in a direct sum of coarse bimodules. 

\begin{lemma}\label{lem: coarse implies mixing}
    If $N\subset M$ is coarse, then $N$ is mixing.
\end{lemma}
\begin{proof}
    Since $ \KKK^{\infty,1}(N) $ is also the weak $N$-$N$ and $JNJ$-$JNJ$ closure of $\KK(L^2N)$, it suffices to show that for any state $\omega \in \B(L^2N)^{\sharp}_J$ (i.e. $\omega$ is normal on both $N$ and $JNJ$) with $\omega|_{\KK(L^2N)} = 0$, we have $\omega( e_NxJyJe_N )=0$ for all $x,y\in M\ominus N$. 

    Since $ L^2(M\ominus N) $ is coarse, there exists a $N$-bimodular unitary $U: L^2(M\ominus N) \to \HH^{\oplus I}$ where $\HH=L^2N\otimes L^2N$ is the coarse bimodule. Consider $ A:= Ux^*e_N: L^2N\to (L^2N)^{\oplus I}\otimes L^2N=\HH$ and  $ B:= UJyJe_N: L^2N\to L^2N\otimes (L^2N)^{\oplus I}=\HH $ then $A$ is right $N$-modular and $B$ is left $N$-modular.

    Fix a unit vector $\xi\in L^2N$, and let $V: L^2N\to L^2N\otimes L^2N$, $ W: L^2N\to L^2N\otimes L^2N $ be the isometries
    $$ V\eta:= \xi\otimes \eta,\quad W\eta:=\eta\otimes \xi,\quad \forall \eta\in L^2N.  $$
    Then since both $A$ and $V^*$ are right $N$-modular, $X:=AV^*\in \B(L^2N,(L^2N)^{\oplus I})\bar{\otimes} N$. Similarly, we have $Y:= BW^*\in  JNJ\bar{\otimes} \B(L^2N,(L^2N)^{\oplus I})$. Let $X_k\in \K(L^2N,(L^2N)^{\oplus I})\bar{\odot} N$ be a net $*$-strongly converging to $ X $, and $ Y_k \in JNJ\bar{\odot} \K(L^2N,(L^2N)^{\oplus I})  $ be a net $*$-strongly converging to $Y$. Take $ A_k:= X_kV $ and $B_k:=Y_kW$. Now we can approximate $ e_NxJyJe_N = A^*B $ by $ A_k^*B_k $. Note that $X_k^*Y_k\in \KK(L^2N\otimes L^2N)$ and therefore $ A_k^*B_k=V^*X_k^*Y_kW\in \KK(L^2N) $.

    Indeed, since $ \omega|_{\KK(L^2N)}=0 $, $\omega( A_k^*B_k )=0$. We have
    \begin{align*}
        &|\omega(  e_NxJyJe_N )|= |\omega(A^*B)|\leq |\omega((A-A_k)^*B_k)|+|\omega(A^*(B-B_k))|\\
        \leq & \omega((A-A_k)^*(A-A_k))^{1/2}\omega(B_k^*B_k)^{1/2}+ \omega( A^*A )^{1/2}\omega( (B-B_k)^*(B-B_k) )^{1/2}.
    \end{align*}
    Since $ N\ni (A-A_k)^*(A-A_k)\to $ strongly, $ \omega((A-A_k)^*(A-A_k)) \to 0$, similarly $\omega( (B-B_k)^*(B-B_k) )\to 0  $. Also, as we can assume that $A_k$ and $B_k$ are uniformly bounded, we obtain $ |\omega(  e_NxJyJe_N )|=0 $.
\end{proof}

\begin{exmp}
    Let $\Gamma$ be an discrete group with a subgroup $H$. We say $ H $ is almost malnormal in $\Gamma$ if for every $ g\in \Gamma\backslash H $, $ gHg^{-1}\cap H $ is finite. Note that this also implies $ g'Hg^{-1}\cap H $ is finite for all $g',g\in \Gamma \backslash H$. For every malnormal subgroup $H$, the group von Neumann subalgebra $L(H)$ is mixing in $L(\Gamma)$ (see for instance \cite{BoutonnetCarderi2017}).
\end{exmp}

\section{Upgrading biexactness}

\subsection{About $ N $-$N$ and $N'$-$N'$ topology on $\B(L^2M)$ for subalgebra $ N\subset M $.}

Suppose $ N\subset M $ is a subalgebra with expectation, to show biexactness of $M$, the related topology is the (weak) $M$-$M$ and $JMJ$-$JMJ$ topology on $\B(L^2M)$. However, if we want to apply the biexactness of $N\subset \B(L^2M)$, the related topology is the (weak) $N$-$N$ and $N'$-$N'$ topology on $ \B(L^2M) $ (see \cite[Lemma 6.5]{ding2023biexact} and the paragraph above the Lemma).

More precisely, by \cite[Lemma 6.5]{ding2023biexact}, biexactness of a von Neumann algebra does not depend on the representation, i.e. if $\HH$ is a normal representation of $N$, define $\KKK(N,\HH)$ to be the $N$-$N$ and $N'$-$N'$ closure of $\K(\HH) $, and
$$ \S(N,\HH):=\{T\in \B(\HH): [T,x]\in \KKK(N,\HH),\forall x\in N'\}. $$
Then the inclusion $N\subset \S(N,\HH)$ is $N$-nuclear iff $N$ is biexact.

In particular, in our cases, we will take $\HH = L^2M$, so the inclusion $N\subset \S(N,L^2M)$ is $N$-nuclear iff $N$ is biexact.

A technical difficulty to relate biexactness of $N$ to biexactness of $M$ is that: $ \KKK(N,L^2M) $ is the $N$-$N$ and $N'$-$N'$ closure of $\KK(L^2M)$, while what we need for $M$ is $ \KKK(M)=\KKK(M,L^2M) $ which is the $ M $-$M$ and $JMJ$-$JMJ$ closure of $\KK(L^2M)$.

In this subsection, we show that in fact the difference between those two topologies can be killed by the projection $ \bar{q}:= \bar{q}_{JMJe_NJMJ}= \vee_{u\in \mathcal{U}(M)} \iota( JuJe_N Ju^*J ) $.

\begin{lemma}\label{lemma for N' topology}
    Let $\bar{q}=\bar{q}_{JMJe_NJMJ }$ be the unit of $ (JMJe_NJMJ)^{**}\subset \BB(L^2M)^{**} $. We have $ [p^{N'}_{\nor},\bar{q}]=[p^{M'}_{\nor},\bar{q}]=0 $, and $ \bar{q}p^{N'}_{\nor} = \bar{q}p^{M'}_{\nor} $.
\end{lemma}
\begin{proof}
    Take again a faithful normal representation $\KK$ of $ \BB(L^2M)^{**} $. We first claim that $ \bar{q}p_{\nor}^{M'}=\bar{q}p_{\nor}^{JMJ} \KK\subset p^{N'}_{\nor}\KK $.
    
    Indeed, note that as $ \bar{q} $ commutes with $ JMJ $ (as $ JMJ $ is in the multiplier of $ JMJe_NJMJ $, or also from the formula $ \bar{q}=\vee_{u\in \mathcal{U}(JMJ)} \iota(ue_Nu^*) $ ), $\bar{q}$ also commutes with $ p^{JMJ}_{\nor} $. In particular, $ \bar{q}p_{\nor}^{JMJ} \KK \subset \overline{ \iota(JMJe_N)p^{JMJ}_{\nor}\KK } $. So to prove the claim, it suffices to show $ \iota(JxJ e_N)\xi \in p^{N'}_{\nor}\K  $ for all $x\in M$ and $\xi\in p^{JMJ}_{\nor}\K$.  Indeed, suppose $ \xi \in p_{\nor}^{JMJ}\KK $, then the vector state $ \varphi_{\xi}\circ \iota $ is normal on $JMJ$. Since for $ n\in JNJ $, $n = \text{Ad}(e_N)(n)+\text{Ad}((e_N)^{\perp})(n)$, we have
    $ \varphi_{\xi}\circ \iota \circ \text{Ad}(e_N) \leq \varphi_\xi\circ\iota $ on $ JNJ $, forcing $  \varphi_{\xi}\circ \iota \circ \text{Ad}(e_N) $ to also be normal on $JNJ$. Now, as $ \varphi_{\iota(JxJe_N)\xi} = \varphi_{\xi}\circ \iota\circ \text{Ad}(e_N)\circ \text{Ad}( Jx^*J ) = \varphi_{\xi}\circ \iota\circ \text{Ad}(e_N)\circ E_{JNJ} \circ \text{Ad}( Jx^*J )$ on $N'=J\langle M,e_N\rangle J$, where $E_{JNJ}:N'=J\langle M,e_N\rangle J\to JNJ$ is the (non-faithful) normal conditional expectation, $ \varphi_{\iota(JxJ e_N)\xi} $ is normal on $N'$. In particular, $  \iota(JxJe_N)\xi\in p^{N'}_{\nor}\KK$.

    From $ \bar{q}p_{\nor}^{M'} \KK\subset p^{N'}_{\nor}\KK \subset p_{\nor}^{M'}\K$, we obtain $ p_{\nor}^{M'}\bar{q}p_{\nor}^{M'} = \bar{q}p_{\nor}^{M'} $. Taking the adjoint, we have $ p_{\nor}^{M'}\bar{q}p_{\nor}^{M'} =\bar{q}p_{\nor}^{M'}= p_{\nor}^{M'}\bar{q} $.
    
    Now that we have $[p_{\nor}^{M'},\bar{q}]=0 $, $ \bar{q}p_{\nor}^{M'} \KK\subset p^{N'}_{\nor}\KK $ then  implies $ p_{\nor}^{M'}\bar{q}\leq p^{N'}_{\nor} $. The rest of the statement follows straightforwordly. 
\end{proof}

The weak $ N $-$N$ and $ N' $-$N'$ topology on $\B(L^2M)$ is given by the predual $$ \left(p^{N}_{\nor}p^{N'}_{\nor}\B(L^2M)^{**}p^N_{\nor}p^{N'}_{\nor}\right)_*. $$ Therefore $ p^{N'}_{\nor}p^{N}_{\nor}\iota(\KKK(N,L^2M) )p^{N}_{\nor}p^{N'}_{\nor} \subset p^{N}_{\nor}p^{N'}_{\nor}\KK(L^2M)^{**} $. Hence, we denote the bidual version of the small-at-infinity boundary piece $ \S(N,L^2M) $ as the following
$$ \tilde{\S}(N,L^2M) := \{ T\in p^{N}_{\nor}p^{N'}_{\nor}\B(L^2M)^{**}p^{N}_{\nor}p^{N'}_{\nor}: [T,\iota(y)]\in p^{N}_{\nor}p^{N'}_{\nor}\KK(L^2M)^{**},\forall y\in N' \}. $$

\begin{lemma}\label{lemma for S(N)}
    Let $ \bar{q} = \bar{q}_{JMJe_NJMJ} $ and $ q=q_{JMJe_NJMJ} = p^{\sharp}_{\nor}\bar{q} $, then
    $$ q\tilde{\S}(N,L^2M)qq_{\KKK}^{\perp}\subset \tilde{\S}(M)q_{\KKK}^{\perp} . $$
\end{lemma}
\begin{proof}
    It suffices to show $q\tilde{\S}(N,L^2M)qq_{\KKK}^{\perp}\subset \iota(JMJ)'$.

    Note that $\iota(JMJ)$ commutes with $q$ and $q^{\perp}_{\KKK}$, we have
    \begin{align*}
        [q\tilde{\S}(N,L^2M)qq_{\KKK}^{\perp}, \iota(JMJ)] &= q[\tilde{\S}(N,L^2M),\iota(JMJ)]qq_{\KKK}^{\perp}\\
       & \subset q(p^{N}_{\nor}p^{N'}_{\nor}\KK(L^2M)^{**})qq_{\KKK}^{\perp}.
    \end{align*}
    By Lemma \ref{lemma for N' topology}, we have $ qp^{N'}_{\nor} = p^{\sharp}_{\nor}\bar{q}p^{N'}_{\nor} = p^{\sharp}_{\nor}\bar{q}p^{JMJ}_{\nor} = p^{\sharp}_{\nor}\bar{q} =q $. On the other hand, as $p^{N}_{\nor}\geq p^{M}_{\nor}\geq p^{\sharp}_{\nor} $, we have $ qp^{N}_{\nor} = \bar{q}p^{\sharp}_{\nor}p^{N}_{\nor} = \bar{q}p^{\sharp}_{\nor} = q $. Plugging in those two identities, we obtain $$ [q\tilde{\S}(N,L^2M)qq_{\KKK}^{\perp}, \iota(JMJ)]\subset q\KK(L^2M)^{**}qq_{\KKK}^{\perp}=0. $$
\end{proof}
 
\begin{lemma}\label{expectation from BH to S}
    If $ N $ is a biexact von Neumann algebra, and $\HH$ is a normal representation of $N$, then there is a $N$-bimodular u.c.p map $ \psi:\B(L^2N)\to\tilde{\S}(N,\HH) $ such that $ \psi|_{N}=\iota^{\sharp}|_{N} $.
\end{lemma}
\begin{proof}
    Since the embeding $ N\subset \S(N,\HH) $ is $N$-nuclear, there exists a net of c.c.p. maps $ \phi_n: N\to \S(N,\HH)  $ factoring through matrix algebras $\mathbb{M}_{k_n}(\C)$ such that $\phi_n$ converges to the inclusion $ N\subset \S(N,\HH) $ in the weak $ N $-$N$ and $N'$-$N'$ topology (as this is weaker than the $N$-$N$ topology). As $\mathbb{M}_{k_n}(\C)$ is injective, each $\phi_n$ extends to a $\tilde{\phi}_n: \B(L^2N) \to {\S}(N,\HH) $. Now, as the weak $ N $-$N$ and $N'$-$N'$ topology is given by the predual $ \left(p^N_{\nor}p^{N'}_{\nor} \B(\HH)^{**}p^N_{\nor}p^{N'}_{\nor}\right)_* $, we can take a weak${}^*$ limit $\psi$ of the net $ \text{Ad}(p^N_{\nor}p^{N'}_{\nor})\circ \iota\circ \tilde{\phi}_n: \B(L^2N)\to \tilde{\S}(N,\HH) $, (since $  \text{Ad}(p^N_{\nor}p^{N'}_{\nor})(\iota( \S(N,\HH) )) \subset \tilde{\S}(N,\HH)$ by definition). This provide us the desired u.c.p map $ \psi: \B(L^2N)\to \tilde{\S}(N,\HH) $, which is $N$-bimodular as $N$ is in the multiplicative domain.
\end{proof}

\subsection{Mixing biexact subalgebra $N\subset M$}
Now fix a mixing biexact subalgebra $N\subset M$ with expectation. We recall the crucial technique in \cite{DKE24} for mixing subalgebras which is generalized in \cite{toyosawa2025relative}.

\begin{lemma}[{\cite[Lemma 5.10]{toyosawa2025relative}}]\label{lem: basic construction in bidual}
    Let $ N\subset M$ be a mixing subalgebra with expectation, and let $q = q_{JMJe_NJMJ} = \bar{q}_{JMJe_NJMJ} p^{\sharp}_{\nor}$. We have
    $$ q\iota^{\sharp}(x)qq^{\perp}_{\KKK}=q\iota^{\sharp}(E_N(x))qq^{\perp}_{\KKK}= \iota^{\sharp}(E_N(x))qq^{\perp}_{\KKK},\quad \forall x\in M. $$
\end{lemma}

Intuitively, this says that $\iota^{\sharp}(M)q_{\KKK}^{\perp} $ and the projection $q$ behave like the basic construction $ \langle M,e_N\rangle $.

We now prove the key technical lemma, which can be considered as a generalization of \cite[Lemma 5.11]{toyosawa2025relative}.
\begin{lemma}\label{key lemma}
    If $N\subset M$ is an biexact mixing subalgebra with expectation, let $\psi: \B(L^2N)\to \tilde{\S}(N,L^2M)$ be a $N$-bimodular u.c.p map as in Lemma \ref{expectation from BH to S}, then the linear map
    \begin{align*}
         \phi_0:Me_N\B(L^2N)e_NM &\to \iota^{\sharp}(M)q \tilde{\S}(N,L^2M) q\iota^{\sharp}(M)q_{\KKK}^\perp \\
         xe_NTe_Ny&\mapsto \iota^{\sharp}(x)q \psi(T) q\iota^{\sharp}(y)q_{\KKK}^\perp
    \end{align*}
    is a $M$-bimodular c.c.p. map, where $ q=:q_{JMJe_NJMJ} $.
    
    In particular, as $ \iota^{\sharp}(M)q  \tilde{\S}(N,L^2M) q\iota^{\sharp}(M)q_{\KKK}^\perp\subset \tilde{\S}(M)q_{\XX_N}q_{\KKK}^{\perp} $, $ \phi_0 $ has a unital normal $M$-bimodular u.c.p extension
    $$ \phi:= (\phi_0^*|_{ \tilde{\S}(M)_*q_{\XX_N}q_{\KKK}^{\perp} }): (Me_N\BB(L^2N)e_NM)^{**} \to \tilde{\S}(M)q_{\XX_N}q_{\KKK}^{\perp}. $$
\end{lemma}
\begin{proof}
    For the second part of the statement, to see that $ \iota^{\sharp}(M)q \tilde{\S}(N,L^2M) q\iota^{\sharp}(M)q_{\KKK}^\perp\subset \tilde{\S}(M)q_{\XX_N}q_{\KKK}^{\perp} $, we simply observe that from Lemma \ref{lemma for S(N)}, $ q \tilde{\S}(N,L^2M) qq^{\perp}_{\KKK} $ commutes with $\iota(JMJ)$, and therefore $ \iota^{\sharp}(M)q \tilde{\S}(N,L^2M) q\iota^{\sharp}(M)q_{\KKK}^\perp $ commutes with $ \iota(JMJ) $. (Also note that $q \leq q_{\XX_N}$.)

    For the first part of the statement, we first show that $\phi_0$ is a contraction. Take an element $ \sum_{k=1}^n x_ke_N T_ke_Ny_k\in Me_N\BB(L^2N)e_NM $. We let $ X\coloneqq ( x_1e_N,\cdots,x_ne_N ) $ and $Y\coloneqq (y_1^*e_N,\cdots, y_n^*e_N)^*$. Consider the polar decomposition of $X$ and $Y$, we have the identity of norms
    \begin{align*}
        &\|\sum_{k=1}^n x_ke_N T_ke_Ny_k\| = \| X\;\text{diag}(T_1,\cdots,T_n) Y^* \|\\ =& \| (X^*X)^{1/2}\;\text{diag}(T_1,\cdots,T_n)(Y^*Y) \|_{\mathbb{M}_n(\B(L^2N))}\\
        =&\| \left[ E_N(x_i^*x_j)\right]_{ij}^{1/2} \;\text{diag}(T_1,\cdots,T_n) \left[ E_N(y_i y^*_j)\right]_{ij}^{1/2}\|_{ \mathbb{M}_n(\B(L^2N)) }.
    \end{align*}

    On the other hand, by Lemma \ref{lem: basic construction in bidual} and the same polar decomposition method, we obtain
    \begin{align*}
        &\|\sum_{k=1}^n \iota^{\sharp}(x_k)q\psi(T_k)q\iota^{\sharp}(y_k)q^{\perp}_{\KKK}\|\\ =& \| \left[ \iota^{\sharp}(E_N(x_i^*x_j))\right]_{ij}^{1/2}\;\text{diag}(q\psi(T_1)q,\cdots,q\psi(T_n)q) \left[ \iota^{\sharp}(E_N(y_iy^*_j))\right]_{ij}^{1/2}(q^{\perp}_{\KKK}\otimes 1)\|\\
        =&\| (q\otimes 1) \left[\left( \psi \otimes 1 \right)\left( \left[ E_N(x_i^*x_j)\right]_{ij}^{1/2} \;\text{diag}(T_1,\cdots,T_n) \left[ E_N(y_i y^*_j)\right]_{ij}^{1/2} \right) \right](q\otimes 1) (q^{\perp}_{\KKK}\otimes 1)\|\\
        \leq& \|\left[ E_N(x_i^*x_j)\right]_{ij}^{1/2} \;\text{diag}(T_1,\cdots,T_n) \left[ E_N(y_i y^*_j)\right]_{ij}^{1/2}\|
    \end{align*}
    where we used the fact that $\psi $ is $N$-bimodular and u.c.p. This shows that $\phi_0$ is a well-defined contraction. The exact same argument shows that $\phi_0$ is also completely contractive.

    Finally, to show that $\phi_0$ is actually completely positive, it suffices to show that its normal extension $\phi: (Me_N\B(L^2N)e_NM)^{**}\to \tilde{\S}(M)q_{\XX_N}q^{\perp_\KKK} $ is unital. For this, we observe that the identity in $ (Me_N\B(L^2N)e_NM)^{**} $ is $ \bar{q}':= \vee_{u\in \mathcal{U}(M)}\iota(ue_Nu^*) $. Therefore,
    \begin{align*}
        \phi( \bar{q}' ) &= \phi( \vee_{u\in \mathcal{U}(M)}\iota(ue_Nu^*) ) = \vee_{u\in\mathcal{U}(M)}\phi_0( ue_Nu^{*} )q^{\perp}_{\KKK}\\ &= q^{\perp}_{\KKK}\vee_{u\in \mathcal{U}(M)}\iota^{\sharp}(u)\psi(1_{L^2N})\iota^{\sharp}(u^*) =  q^{\perp}_{\KKK}\vee_{u\in \mathcal{U}(M)}\iota^{\sharp}(u)q\iota^{\sharp}(u^*)\\
        &=q^{\perp}_{\KKK}\vee_{u\in \mathcal{U}(M)} \iota^{\sharp}(u)\left(\vee_{v\in \mathcal{U}(M)} \iota^{\sharp}(JvJe_N Jv^*J)\right)\iota^{\sharp}(u^*) = q^{\perp}_{\KKK}q_{\XX_N},
    \end{align*} 
which is indeed the unit of $ \tilde{\S}(M)q_{\XX_N}q_{\KKK}^{\perp} $.   
\end{proof}

\subsection{Proof of upgrading biexactness}
\begin{corollary}\label{cor: biexact mixing subalg}
    If $M$ is weakly exact, and $N\subset M$ is a biexact mixing subalgebra with expectation, then the following map is weak${}^{*}$-nuclear:
    $$  \text{Ad}(q_{\XX_{{N}}}q_{\KKK}^{\perp})\circ \iota^{\sharp}|_M :M\to \tilde{\S}(M)q_{\XX_{{N}}}q_{\KKK}^{\perp}. $$
\end{corollary}
\begin{proof}
    Since $M$ is weakly exact, the embedding $ M\subset \B(L^2M) $ is $M$-nuclear (\cite[Theorem 5.1]{ding2023biexact}). Since $\bar{q}':= \vee_{u\in \mathcal{U}(M)}\iota(ue_Nu^*)$ is the identity of the bidual of the hereditary $C^*$-subalgebra $ \overline{ Me_N\B(L^2N)e_N M }\subset \B(L^2M) $, we have the identity $ \bar{q}'\B(L^2M)^{**}\bar{q}'  = (Me_N\B(L^2N)e_N M)^{**}$, which gives us the $M$-bimodular map
    $$ \text{Ad}(\bar{q}')\circ \iota: \BB(L^2M)\to (Me_N\B(L^2N)e_N M)^{**}.$$
    Now, let $\phi:  (Me_N\B(L^2N)e_N M)^{**} \to \tilde{\S}(M)q_{\XX_N}q^{\perp}_{\KKK} $ be the u.c.p normal map in Lemma \ref{key lemma}. We have the following desired composition map
    $$ M \subset \B(L^2M)\xrightarrow{ \text{Ad}(\bar{q}')\circ \iota} (Me_N\B(L^2N)e_N M)^{**}  \xrightarrow{ \phi}\tilde{\S}(M)q_{\XX_N}q^{\perp}_{\KKK}.  $$

    We claim that this composition map is weak${}^*$-nuclear.
    Indeed, recall that the inclusion $M\subset \B(L^2M)$ is $M$-nuclear, but the weak $M$-$M$ topology is given by the predual of $ p^M_{\nor}\B(L^2M)p^M_{\nor} $, therefore we have that the map
    $$ \text{Ad}(\bar{q}')\circ\text{Ad}(p^{M}_{\nor})\circ  \iota=\text{Ad}(p^{M}_{\nor})\circ \text{Ad}(\bar{q}')\circ \iota: M\to p^{M}_{\nor}(Me_N\B(L^2N)e_N M)^{**}p^{M}_{\nor} $$
    is weak${}^*$-nuclear. But since $ \phi = \phi\circ \text{Ad}(p^{M}_{\nor}) $ (as $ p^{\sharp}_{\nor}\leq p^{M}_{\nor} $), our desired composition $ \phi\circ \text{Ad}(\bar{q}')\circ \iota $ can also be written as $ \phi\circ \text{Ad}(p^{M}_{\nor})\circ \text{Ad}(\bar{q}')\circ \iota $, which is also weak${}^*$-nuclear as $\phi$ is normal.

    Finally, we claim that the composition map $\phi\circ \text{Ad}(\bar{q}')\circ \iota$ is precisely $ \text{Ad}(q_{\XX_N}q_{\KKK}^{\perp})\circ \iota^{\sharp}|_M $ when restricted to $M$, which finishes the proof. Indeed, since all the maps involved are $M$-bimodular, we only need to check it at the identity:
    $$ \phi\circ \text{Ad}(\bar{q}')\circ \iota(1) = \phi(\bar{q}') = q_{\XX_N}q_{\KKK}^{\perp}, $$
    where the last identity holds, as $\phi$ is unital as in the proof of Lemma \ref{key lemma}.
\end{proof}

\begin{thm}\label{thm: main}
    Let $M$ be a separable von Neumann algebra and $\mathcal{N}=\{N_i\}_{i\in I}$ be a family of subalgebras with expectation. Let $q_{\XX_{N_i}}\in \B(L^2M)^{\sharp *}_{J}$ be the identity of $ (\XX_{N_i})^{\sharp *}_J $. Assume that $[ q_{\XX_{N_i}},q_{\XX_{N_j}} ]=0$ for all $i,j$. If $M$ is biexact relative to $\mathcal{N}$, and each $N_i$ is mixing and biexact, then $M$ is biexact.
\end{thm}
\begin{proof}
    We first follow the same argument as in the proof of \cite[Lemma 5.20]{toyosawa2025relative}. Since $M$ is separable and biexact relative to $\mathcal{N}$, by Theorem \ref{thm: bidual characterization of biexact}, the embedding $\iota^{\sharp}|_M: M\to \tilde{\S}_{\XX_{{\mathcal{N}}}}(M) $ is weak${}^*$-nuclear. 
    
    Composing $ \text{Ad}(q_{\mathbb{K}}) $ and $\text{Ad}(q_{\XX_{\mathcal{N}}}^{\perp})$ with this weak${}^*$-nuclear embedding, we obtain two weak${}^* $-nuclear maps $$\text{Ad}(q_{\mathbb{K}})\circ\iota^{\sharp}|_{M}: M\to \K(L^2M)^{\sharp *}_J,\quad \text{Ad}(q_{\mathbb{X}_{\mathcal{N}}}^{\perp})\circ\iota^{\sharp}|_{M}: M\to q_{\mathbb{X}_{\mathcal{N}}}^{\perp}\tilde{\mathbb{S}}_{\mathbb{X}_{\mathcal{N}}}(M)q_{\mathbb{X}_{\mathcal{N}}}^{\perp}. $$
    Note that both maps have their image contained in $ \tilde{\S}(M) $. Again by Theorem \ref{thm: bidual characterization of biexact}, to show the biexactness of $ M $, it suffices to show that the inclusion $\iota^{\sharp}|_M: M\to \tilde{\S}(M) $ is weak${}^*$ nuclear. As $ \iota^{\sharp}|_M = \text{Ad}(q_{\mathbb{X}_{\mathcal{N}}}^{\perp})\circ \iota^{\sharp}|_M+ \text{Ad}(q_{\KKK})\circ \iota^{\sharp}|_M + \text{Ad}(q_{\XX_{\mathcal{N}}}q_{\KKK}^{\perp})\circ \iota^{\sharp}|_M $ and the first two maps in the sum are weak${}^*$-nuclear, we only need to show that the third map
    $$ \text{Ad}(q_{\XX_{\mathcal{N}}}q_{\KKK}^{\perp})\circ \iota^{\sharp}|_M :M\to \tilde{\S}(M)q_{\XX_{\mathcal{N}}}q_{\KKK}^{\perp} $$
    is weak${}^*$-nuclear.

    By Corollary \ref{cor: biexact mixing subalg}, for each $i\in I$, $\text{Ad}(q_{\XX_{{N}_i}}q_{\KKK}^{\perp})\circ \iota^{\sharp}|_M :M\to \tilde{\S}(M)q_{\XX_{{N}_i}}q_{\KKK}^{\perp}$ is weak$^*$-nuclear.  Recall that $ q_{\XX_{\mathcal{N}}} = \vee_{i\in I} q_{\XX_{N_i}}  $. Since we assume $[ q_{\XX_{N_i}},q_{\XX_{N_j}} ]=0$, we can glue up the maps $\text{Ad}(q_{\XX_{{N}_i}}q_{\KKK}^{\perp})\circ \iota^{\sharp}|_M$ as follows: For different $i,j$, since $ \text{Ad}(q_{\XX_{{N}_j}}q_{\XX_{{N}_i}}q_{\KKK}^{\perp})= \text{Ad}(q_{\XX_{{N}_j}})\circ\text{Ad}(q_{\XX_{{N}_i}}q_{\KKK}^{\perp})$, we obtain the weak${}^*$-nuclear map $$\text{Ad}(q_{\XX_{{N}_j}}q_{\XX_{{N}_i}}q_{\KKK}^{\perp})\circ \iota^{\sharp}|_M: M\to \tilde{\S}(M)q_{\XX_{\mathcal{N}}}q_{\KKK}^{\perp}.$$
    But since all the involved projections commute with $M$, we have $$\text{Ad}\left((q_{\XX_{{N}_i}}\vee q_{\XX_{N_j}})q_{\KKK}^{\perp}\right) = \text{Ad}(q_{\XX_{{N}_i}}q_{\KKK}^{\perp})+\text{Ad}(q_{\XX_{{N}_j}}q_{\KKK}^{\perp})- \text{Ad}(q_{\XX_{{N}_j}}q_{\XX_{{N}_i}}q_{\KKK}^{\perp})$$ when restricted to $\iota^{\sharp}(M)$. In particular, $\text{Ad}\left((q_{\XX_{{N}_i}}\vee q_{\XX_{N_j}})q_{\KKK}^{\perp}\right)\circ \iota^{\sharp}|_{M}:M\to \tilde{\S}(M)q_{\XX_{\mathcal{N}}}q_{\KKK}^{\perp}  $ is weak${}^*$-nuclear. Inductively, we can use the same argument to replace $ q_{\XX_{{N}_i}}\vee q_{\XX_{N_j}} $ above with $ \vee_{i\in \mathcal{F}} q_{\XX_{N_{i}}}  $ for any finite subset $\mathcal{F}\subset I$. Since $ q_{\XX_{\mathcal{N}}} $ is the $\sigma$-strong limit of $ \vee_{i\in \mathcal{F}} q_{\XX_{N_{i}}} $, $  \text{Ad}(q_{\XX_{\mathcal{N}}}q_{\KKK}^{\perp})\circ \iota^{\sharp}|_M $ is the point-weak$^*$ limit of $ \text{Ad}(q_{\KKK}^{\perp}\vee_{i\in \mathcal{F}\subset I} q_{\XX_{N_{i}}})\circ \iota^{\sharp}|_M $ and therefore also weak$^*$-nuclear.
    
\end{proof}

We note that the assumption $ [q_{\XX_i},q_{\XX_j}]=0 $ is always satisfied when dealing $L\Gamma$ with subgroups von Neumann algebras $\{LH_i\}_{i\in I}$.
\begin{corollary}
    Let $G$ be a countable discrete group biexact relative to a family of subgroups $ \{H_i\}_{i\in I} $. If each $H_i$ is biexact, and almost-malnormal in $G$, then $G$ is biexact.
\end{corollary}
\begin{proof}
    Since $H_i$ is almost-malnormal in $G$, $L(H_i)$ is mixing in $L(G)$. By Theorem \ref{thm: main} and Theorem \ref{thm: DP 6.2}, it remains to check that $ q_{L(H_i)} $ commutes with each other. For this, we note that $ q_{L(H_i)} = p^{\sharp}_{\nor}\vee_{g,g'\in H_i}\iota( J\lambda_{g'}J\lambda_ge_{L(H_i)} \lambda_g^*J\lambda_{g'}^*J) $ and $ J\lambda_{g'}J\lambda_ge_{L(H_i)} \lambda_g^*J\lambda_{g'}^*J $ is simply the projection onto $ g\ell^2(H_i)g' $. Therefore, for each $i,j$, $g,g'\in H_i$, $h,h'\in H_j$, we have $$ [ J\lambda_{g'}J\lambda_ge_{L(H_i)} \lambda_g^*J\lambda_{g'}^*J), J\lambda_{h'}J\lambda_he_{L(H_j)} \lambda_h^*J\lambda_{h'}^*J)] =0.$$ As all these projections commute with $ p^{\sharp}_{\nor} $, we obtain $[q_{L(H_i)} ,q_{L(H_j)} ]=0 $.
\end{proof}

As an application, using \cite[Proposition 4.15]{toyosawa2025relative}, we obtain a sufficient condition for the biexactness of amalgamated free product algebras.
\begin{corollary}\label{cor: biexact AFP}
    Let $M_1,M_2$ be two separable biexact tracial von Neumann algebras with a common subalgebra $B$ with expectation. If $B$ is amenable and mixing in both $M_1$ and $M_2$, then the amalgamated free product $ M=M_1\ast_BM_2 $ is biexact.
\end{corollary}
\begin{proof}
    By \cite[Proposition 4.15]{toyosawa2025relative}, $M$ is biexact relative to $\{M_1,M_2\}$. Also, since $B$ is mixing in $M_i$, both $M_1$ and $M_2$ are mixing in $M$ (\cite[Lemma 5.5]{toyosawa2025relative}). It remains to check that $ q_i= q_{\XX_{M_i}} $ commute with each other. 
    
    For simplicity, denote $e_i=e_{M_i}$. Let $\K$ be a normal representation of $\B(L^2M)^{\sharp *}_J$.  By the lemma below, it suffices to show that $\iota^{\sharp}(e_1)q_2\KK =\overline{\iota^{\sharp}( e_1MJMJ e_2)\KK} \subset \overline{\iota^\sharp( MJMJe_B)\KK} = q_{\XX_B}\KK$.
    
    Note that \begin{align*}
        e_1MJMJ e_2 =& e_1(M\ominus M_1)J(M\ominus M_1)J e_2+e_1M_1J(M\ominus M_1)J e_2 + e_1MJM_1J e_2\\
        =& e_1(M\ominus M_1)J(M\ominus M_1)J e_2+M_1e_1J(M\ominus M_1)J e_2\\ &+ JM_1Je_1(M\ominus M_1)e_2+ JM_1JM_1e.
    \end{align*}
    So, it suffices to show that the range of these operators belongs to $\iota^\sharp( MJMJe)\KK$.
    
    For this, we first claim that $  (e_1-e_B)(M\ominus M_1)J(M\ominus M_1)J e_2=0$, or equivalently $e_2(M\ominus M_1)J(M\ominus M_1)J (e_1-e_B)=0$. Indeed, a vector in $(M\ominus M_1)J(M\ominus M_1)J (e_1-e_B)L^2M$ is orthogonal to $L^2M_2$. Similarly, we have $ (e_1-e_B)J(M\ominus M_1)J e_2 =(e_1-e_B)(M\ominus M_1)e_2 =0$.
    
    Therefore, we have $ \iota^{\sharp}(e_1(M\ominus M_1)J(M\ominus M_1)J e_2)\KK \subset \iota^\sharp( e_B(M\ominus M_1)J(M\ominus M_1)J e_2)\KK \subset \iota^\sharp( MJMJe_B)\KK  $. Similarly, we have $ M_1e_1J(M\ominus M_1)J e_2, JM_1Je_1(M\ominus M_1)e_2\subset \iota^\sharp( MJMJe_B)\KK $. Combining these inclusions, we obtain $ \iota^{\sharp}( e_1MJMJ e_2)\KK \subset \iota^\sharp( MJMJe_B)\KK $.
\end{proof}

\begin{lemma}\label{lem: checking commuting}
    Suppose that $ M_1,M_2 $ are two subalgebras of $M$ with expectation and with $ e_{i}= e_{M_i} $ and $ q_{i}=q_{\XX_i} $, and $ B\subset M_i$ is a subalgebra with expectation. If $$ \iota^{\sharp}(e_1)q_2\KK =\overline{\iota^{\sharp}( e_1MJMJ e_2)\KK} \subset \overline{\iota^\sharp( MJMJe_B)\KK} = q_{\XX_B}\KK,$$ then $ [q_1,q_2]=0$ and $q_1q_2=q_{\XX_B}$.
\end{lemma}
\begin{proof}
    Since $q_2$ commutes with $\iota^{\sharp}(M)$ and $\iota^{\sharp}(JMJ)$, and $q_1 = \vee_{u,v\in \mathcal{U}(M)} \iota^{\sharp}(uJvJe_1Jv^*Ju^*)$, to show $[q_1,q_2] = 0$, it suffices to check that $ [\iota^{\sharp}(e_1),q_2] = 0 $. Note that by assumption, we have $ \iota^{\sharp}(e_1)q_2 = q_{\XX_B}\iota^{\sharp}(e_1)q_2 = q_2\iota^{\sharp}(e_1)q_2 $. Taking the adjoint, we obtain $ \iota^{\sharp}(e_1)q_2=q_2\iota^{\sharp}(e_1)q_2=  \iota^{\sharp}(e_1)q_2 $.

    Note that as $ [\iota^{\sharp}(e_1),q_2] =0$, we have
    $$ q_1q_2 =q_2q_1= \vee_{u,v\in \mathcal{U}(M)} \iota^{\sharp}(uJvJe_1Jv^*Ju^*)q_2 = \vee_{u,v\in \mathcal{U}(M)} \iota^{\sharp}(uJvJ) \iota^{\sharp}(e_1)q_2\iota^{\sharp}(Jv^*Ju^*). $$
    Since $ \iota^{\sharp}(e_1)q_2 = q_{\XX_B}\iota^{\sharp}(e_1)q_2 = q_2\iota^{\sharp}(e_1) = q_2\iota^{\sharp}(e_1)q_{\XX_B}= q_{\XX_B}\iota^{\sharp}(e_1)q $, we obtain
    \begin{align*}
        q_1q_2 =& \vee_{u,v\in \mathcal{U}(M)} \iota^{\sharp}(uJvJ) q_{\XX_B}\iota^{\sharp}(e_1)q_{\XX_B}\iota^{\sharp}(Jv^*Ju^*)\\ =& q_{\XX_B}\left( \vee_{u,v\in \mathcal{U}(M)} \iota^{\sharp}(uJvJe_1Jv^*Ju^*) \right)q_{\XX_B}= q_{\XX_B}q_1q_{\XX_B}=q_{\XX_B}.
    \end{align*} 
\end{proof}

\section{Application to graph product von Neumann algebras}
We will always assume a graph $ \GG=(\VV(\GG),\EE) $ is a simple (possibly infinite) graph, i.e. it is an undirected graph without multiple edges and without edges from a vertex to itself. For two vertices $u,v\in \VV(\GG)$, we denote $ u\sim v $ if $ (u,v)\in \EE $. We say that a graph $\GG_1=(\VV(\GG_1),\EE_1)$ is a subgraph of $\GG$ if $ \VV(\GG_1)\subset \VV(\GG) $, and $ u\sim v $ in $ \GG_1 $ if and only if $ u\sim v $ in $\GG$. For each $v\in \VV(\GG)$, we consider the link subgraph $\text{link}_v$ with $ \VV(\text{link}_v) = \{u\in \VV(\GG):u\sim v\} $ and the star subgraph $ \text{star}_v $ with $ \VV(\text{star}_v) = \{v\}\cup\{u\in \VV(\GG):u\sim v\} $. A complete subgraph of $\GG$ is called a clique.

Let $ \{( M_v,\varphi_v ):v\in \VV(\GG)\} $ be a family of von Neumann algebras with faithful normal states. Consider the graph product von Neumann algebra $(M,\varphi)= \ast_{v\in \GG} (M_v,\varphi_v)$. The $L^2$ space of $(M,\varphi)$ can be written the direct sum of Hilbert subspaces $ \HH_{v_1\cdots v_n}= \H_{v_1}^{o}\otimes \cdots \otimes \H^o_{v_n} $ where $ v_1\cdots v_n $ is a reduced word in $ \GG $ and $ \H^o_{v_n}:=L^2( M_v,\varphi_v )\ominus \C\hat{1} $. See, for example, \cite{charlesworth2025structure} for a detailed definition of graph product von Neumann algebras.

\subsection{Subgraph boundary piece of graph products}

In this subsection, we check that the commuting assumption $ [q_{\XX_{N_i}},q_{\XX_{N_i}}]=0 $ is satisfied for graph product subalgebras of subgraphs.

Let $ \GG_1 $ and $\GG_2$ be two subgraphs, and $M_i:= \ast_{v\in \GG_i}(M_v,\varphi_v)$, considered as subalgebras of $M$. Let also $N = M_1\cap M_2 = \ast_{v\in \GG_1\cap \GG_2}(M_v,\varphi_v)$ be the graph product over the intersection subgraph $\GG_1\cap \GG_2$. 

Let $e_i$, $i=1,2$ be the projection onto the subspace $L^2M_i \subset L^2M$. And $e$ be the projection onto $L^2N$. Let also $ q_{i}=q_{\XX_i} $ and $q=q_{\XX_N}$. We can use the same argument as in the proof of Corollary \ref{cor: biexact AFP} to show that $ [q_1,q_2]=0 $ and $q_1q_2=q$.

\begin{lemma}
    For the subgraph product algebras $M_1$ and $M_2$, we have $ [q_1,q_2]=0 $ and $q_1 q_2 =q$.
\end{lemma}
\begin{proof}
     Let $ \KK $ be a normal representation of $\BB(L^2M)^{\sharp*}_J$ so that $ \BB(L^2M)^{\sharp*}_J\subset \B(\KK) $. By Lemma \ref{lem: checking commuting}, we need to show that $$\iota^{\sharp}(e_1)q_2\KK =\overline{\iota^{\sharp}( e_1MJMJ e_2)\KK} \subset \overline{\iota^\sharp( MJMJe)\KK} = q\KK.$$    

    For this, it suffices to show that $ \iota^{\sharp}(e_1xJyJe_2)\KK\subset q\KK $ for $x,y$ being reduced words in $ M^o_{v} $'s with $v\in \GG$. But notice that if $x$ has reduced form $ x=x_1\cdots x_n $ begining with $x_1\in M^o_{v}$ with $v\in \GG_1$, then $ \iota^{\sharp}(e_1xJyJe_2)\KK = \iota^{\sharp}(x_1)\iota^{\sharp}(e_1x_2\cdots x_n JyJe_2)\K $. Since $ \iota^{\sharp}(x_1)q\KK \subset q\KK$, it reduced to the case when $x$ does not begin with a element in $ M^o_{v} $ with $v\in \GG_1$. Repeat the same reduction for $JyJ$, we may without loss of generality assume that the reduced words $x$ does not have a reduced form begining with $x_1\in M^o_v$, $v\in \GG_1$, and that $ JyJ $ does not have a reduced form begining with $Jy_1J \in JM^0_vJ$, $v\in \GG_1$.

     Now with this additional assumption, we claim that $ (e_1-e)xJyJe_2=0 $. Indeed, it suffices to show $ e_2Jy^*Jx^*(e_1-e)=0 $, or equivalently $ e_2Jy^*Jx^*\xi=0 $ for $\xi \in \HH_{v_1\cdots v_n} $ with $v_1\cdots v_n$ reduced, $v_i\in \GG_1$ for all $i$, and $v_{i_0}\notin \GG_1\cap \GG_2$ for some $i_0$. For simplicity, write $ \xi=\xi_1\otimes\cdots \otimes\xi_n $ with $\xi_i\in L^2M^o_{v_i}$, and write $ Jy^*Jx^*\xi $ as a linear combination of reduced words. We now notice that our assumption on $ x $ and $JyJ$ imply that the element $\xi_{i_0}$ is not canceled in every reduced word in $ Jy^*Jx^*\xi $. Therefore, $  Jy^*Jx^*\xi $ is orthogonal to $ L^2M_2 $. In particular, this implies $ e_2Jy^*Jx^*\xi=0 $.

     Finally, $  \iota^{\sharp}(e_1xJyJe_2)\KK = \iota^{\sharp}((e_1-e)xJyJe_2)\KK+ \iota^{\sharp}(exJyJe_2)\KK=\iota^{\sharp}(exJyJe_2)\KK\subset q\KK$.
\end{proof}

\subsection{Relative biexactness of graph products}
In \cite{hoshino2026relative}, it is proven that the graph product of biexact groups is biexact relative to the family of link graph products. We generalize this to von Neumann algebras. 
\begin{thm}\label{thm: relative biexact graph product}
     If $\GG$ is a graph and $M_v$ is biexact for each $v\in \mathcal{V}(\GG)$, then the graph product von Neumann algebra $ M=\ast_{v\in \GG}(M,\varphi_v)$ is biexact relative to the family $ \{ M_{\text{link}_v} \}_{v\in \VV}$, where $  M_{\text{link}_v} = \ast_{v\in \text{link}_v} (M_v,\varphi_v) $.
\end{thm}
\begin{proof}
    We follow the proof of \cite[Theorem 5.10]{ding2023biexact} and the proof of \cite[Theorem A]{hoshino2026relative}.

    For each $v$, since $M_v$ is biexact, the embedding $ M_v\to \S(M_v)\subset \B(L^2M_v) $ is $M_v$-nuclear for each $v$. By \cite[Lemma 5.9]{ding2023biexact}, there exists two nets of state preserving u.c.p. maps $\phi^v_{k}:(M_v,\varphi_v)\to (\M_{n(k,v)}(\C),\mu_{k}^v)$, $ \psi^v_k:(\M_{n(k,v)}(\C),\mu_{k}^v)\to \S(M_v)  $ such that the composition $ \theta^v_k:= \psi^v_k\circ \phi^v_{k} $ satisfies that for each $x\in M_v$, $ \theta^v_k(x)-x $ can be written as
    \begin{equation}\label{eq 1}
        \theta^v_k(x)-x= T_k^v+x_k^v,
    \end{equation}  
    with $x_k^v\in M_v$, $ \varphi_v((x_k^v)^*x_k^v)\to 0 $, $T_k^v\in \S(M_v)$, $\|T_k^v\|\to 0$, and $ \mu_k^v $ is a pure state on $ \M_{n(k,v)}(\C) $.

    Recall that $ \S(M_v) \subset \B(L^2M_v)$, we abuse the notation and also denote by $\varphi_v$ the vector state given by $ \bar{1}_{\varphi_v}\in L^2(M_v,\varphi_v) $ on $ \B(L^2M_v) $. Consider the reduced graph product $S$ of the operator systems $(\S(M_v),\varphi_v)$, i.e. $S$ is the operator system contained in the graph product of $(\B(L^2M_v),\varphi_v)$ generated by elements whose reduced words only consists of operators from $ \S(M_v)$. Consider also $M_0$ the reduced $C^*$-graph product of $(M_v,\varphi_v)$. By \cite{atkinson2019graph}, there exists two nets of u.c.p. maps $ \phi_k: M_0 \to B_k $ and $\psi_k: B_k\to S$, which are the reduced graph products of $ \phi_k^v $ and $\psi^v_k$. Here $B_k$ is the reduced $C^*$ graph product of $(\M_{n(k,v)}(\C),\mu_{k}^v)$.

    Let $ \theta_k:= \psi_k\circ\phi_k$. We claim that for any reduced word $ x=a_1a_2\cdots a_n \in M_0 $ with $a_i\in M_{v_i}$ and $\varphi_{v_i}(a_i)=0$, $\theta_k(x)\to x$ in the weak ($M_0\subset M$)-topology (\cite[Section 3]{ding2023biexact}). Indeed, note that $ \theta_k(x)-x $ is spanned by terms of the form $ \theta_k^{v_1}(a_1)\cdots \theta_k^{v_{m-1}}(a_{m-1})\left(\theta_k^{v_m}(a_m)-a_m \right)a_{m+1}\cdots a_n $. But for any state $ \varphi\in S^{(M_0\subset M)\sharp (M_0\subset M)} $, we have by equation \eqref{eq 1},
    $$ \varphi\left( \theta_k^{v_1}(a_1)\cdots \theta_k^{v_{m-1}}(a_{m-1})\left(\theta_k^{v_m}(a_m)-a_m \right)a_{m+1}\cdots a_n \right) \to 0,$$
    which implies that $\theta_k(x)\to x$ in the weak ($M_0\subset M$)-topology. Finally, since each $\theta_k: M_0\to S$ factors through the nuclear algebra $B_k $ \cite[Proposition 3.1]{hoshino2026relative} (since it is the graph products of matrix algebras with pure states), the embedding $ M_0\subset S $ is $(M_0\subset M)$-nuclear.

    Next, we claim that the reduced graph product $S$ is contained in the small-at-infinity $\S_{\XX}(M)$ with $\XX$ the boundary piece generated by the family of subalgebras $ \{M_{\text{link}_v} \}_{v\in \VV}$.

    Indeed, by \cite[Lemma 6.1]{DKEP23} and \cite[pg. 13]{ding2023biexact}, we only need to check this on generators of $M$: Let $ T\in \S(M_v) $, $a\in M_u$. Taking an element $\xi = \xi_1\otimes \cdots \otimes \xi_n \subset L^2M_{w_1}\otimes L^2M_{w_2}^o\otimes \cdots \otimes L^2M_{w_n}^o$ for an irreducible word $ w_1\cdots w_n $, we can check that $[T,JaJ]\xi  $ is nonzero only when $ v=u=w_1 $ and $w_2,\cdots,w_n\in \text{link}(w_1)$. Therefore, we have 
    \[[T,JaJ]\xi = \begin{cases}
        ([T,JaJ]\xi_1)\otimes \xi_2\otimes\cdots \otimes \xi_n, \quad &v=u=w_1,w_2,\cdots,\quad w_n\in \text{link}(w_1);\\
        0,\quad &\text{otherwise}.
    \end{cases} \]
    In particular, $[T,JaJ]\in p\KKK^{\infty,1} (M_{v})\otimes \B(L^2 M_{\text{link}_v} )p\subset \KKK^{\infty,1}_{\XX}(M)$, where $p$ is the projection onto $L^2M_{\text{star}_v}$.
    
    Finally, since $M_0$ is ultraweakly dense in $M$, and the embedding $M_0\subset S\subset  \S_{\XX}(M)$ is $(M_0\subset M)$-nuclear, by \cite[Corollary 4.9]{ding2023biexact}, the embedding $ M\subset \S_{\XX}(M) $ is $M$-nuclear.
\end{proof}

\begin{remark}
    Note that if we only assume that each $M_v$ is weakly exact, then the same proof shows that $ M=\ast_{v\in \GG}(M_v,\varphi_v) $ is biexact relative to the family $\{ M_{\text{star}_v} \}_{v\in \VV}$.
\end{remark}

\begin{thm}
    Let $\GG$ be a simple graph without an infinite clique. The following are equivalent:
    \begin{enumerate}
        \item For any family of finite dimensional von Neumann algebras with faithful normal states $\{(M_v,\varphi_v)\}_{v\in \VV(\GG)}$ such that $M_v\neq \C$, the graph product von Neumann $M=\ast_{v\in \GG}(M_v,\varphi_v)$ is biexact.
        \item There is no square in $\GG$: there do not exist distinct vertices $v_1,v_2,u_1,u_2\in \VV(\GG)$ such that $ v_1\nsim v_2$, $u_1\nsim u_2$, and $ v_i\sim u_j $ for each $i,j=1,2$.
    \end{enumerate}
\end{thm}
\begin{proof}
    (2) $\implies$ (1): We first claim that $M$ is biexact if $ M_{\text{link}_v} $ is biexact for all $v$.
    
    Indeed, if there is a $v_0\in \VV(\GG)$ such that $ \text{star}(v_0)= \GG $, then $ M= M_{\text{link}_{v_0}}\otimes M_{v_0} $. Therefore $M$ is biexact iff $ M_{\text{link}_{v_0}} $ is biexact. Now assume no such $v_0$ exists. By \cite[Proposition 5.5]{charlesworth2025structure} (the same proof works for infinite graphs), for each $v\in \VV(\GG)$, our assumption implies that the $ M_{\text{star}_v} $-bimodule $ L^2M\ominus L^2M_{\text{star}_v} $ can be written as a direct sum of bimodules of the form $ L^2M_{\text{star}_v}\otimes_{D}L^2M_{\text{star}_v} $ with $D $ the graph product algebra of a subgraph of $\text{link}_v$ with at most two (adjacent) vertices. As $D$ is finite dimensional, $  L^2M_{\text{star}_v}\otimes_{D}L^2M_{\text{star}_v} $ is contained in a direct sum of coarse bimodules. In particular, $ M_{\text{star}_v} $ is coarse and therefore mixing by Lemma \ref{lem: coarse implies mixing}. Since $ M $ is biexact relative to $ \{ M_{\text{link}_v} \}_{v\in \VV} $ by Theorem \ref{thm: relative biexact graph product}, it is also biexact relative to $\{ M_{\text{star}_v} \}_{v\in \VV}$. It then follows from Theorem \ref{thm: main} that $M$ is biexact whenever $ M_{\text{link}_v} $ is biexact for each $v$ as $ M_{\text{link}_v}\otimes M_v = M_{\text{star}_v} $.

    Now, suppose for contradiction that $M$ is not biexact, then there exists a $v_1$ such that $ M_{\text{link}_{v_1}} $ is not biexact. But the link subgraph of $ \text{link}_{v_1} $ is of the form $ \text{link}_{v_1}\cap \text{link}_{u} $ with $u\in \text{link}_{v_1}$. Therefore, there exists $v_2\in \text{link}_{v_1}$ such that $ M_{\text{link}_{v_1}\cap M_{\text{link}_{v_2}}} $ is not biexact. Repeating this process, we obtain a sequence $ v_1,\cdots,v_n,\cdots $ such that $ M_{\bigcap_{j\leq n}\text{link}_{v_j}} $ is not biexact for each $n$. But now $\{v_j\}_{j=1}^\infty$ form an infinite clique, contradicting our assumption.

    (1) $\implies$ (2): Suppose (2) does not hold and there exist such vertices $v_1,v_2,u_1,u_2$. Choose $ M_{v_i}= M_{u_i}=\M_2(\C)$ with the canonical trace, then $ M $ contains the two commuting subalgebras $ M_{u_1}\ast M_{u_2} $ and $M_{v_1}\ast M_{v_2}$ which are nonamenable $\text{II}_1$ factors. This implies that $M$ is not solid and therefore not biexact.
\end{proof}

\section{Relatively biexact groups over almost malnormal subgroups}\label{relatively biexact section end}
In this section, we isolate two self-contained proofs for the specific case of our main Theorem \ref{main intro}, in the context of countable groups and almost malnormality. The first proof uses the characterization of the relative biexactness of a group $\Gamma$ via the action of $\Gamma \times \Gamma$. The second proof uses a similar strategy using the bidual, as in the proof of our main theorem.

Let $\Gamma$ be a discrete group, and $  \{ H_i\}_{i\in I}$ be a family of subgroups $ \Gamma $. For a function $f\in \ell^\infty(\Gamma)$ and $s\in \Gamma$, we always denote the left translation by $s\cdot f$, and the right translation by $\rho_s(f)$. For a probability measure $\mu \in \text{Prob}(\Gamma)$, we also denote by $s\cdot m$ the measure $ s\cdot m(\{g\}) = m(\{s^{-1}g\})$.

A $\Gamma$-boundary piece is a left and right $\Gamma$-invariant closed ideal $I\subset \ell^\infty(\Gamma)$. In particular, we are interested in the subgroups boundary piece $I=  c_0(\Gamma,\{H_i\}_{i\in I})$ consisting of functions $f\in \ell^\infty(\Gamma)$ vanishing at $\infty/\{H_i\}_{i\in I}$, i.e., for each $\varepsilon>0$, $ \{g\in \Gamma: |f(g)|>\varepsilon\} $ is contained in a finite union $ \bigcup gHg' $ with $g,g'\in \Gamma$.

The relative small-at-infinity boundary of $\Gamma$ is defined as the subalgebra
\[ \S_I(\Gamma):= \{f\in \ell^\infty(\Gamma): f-\rho_t(f)\in I,\forall t\in \Gamma\}.\]
And $\Gamma$ is said to be biexact relative to $I$ if the left action of $\Gamma$ on $\S_I(\Gamma)/I$ (or equivalently $ \S_I(\Gamma)$ \cite[Section 6]{ding2023biexact}) is amenable. For more details regarding the biexactness of groups, we refer the reader to the standard text \cite[Section 15.2]{brown2008textrm}

\subsection{A proof via the action of $\Gamma\times \Gamma$}
We recall that $\Gamma$ is biexact relative to $I$ iff the left and right action of $ \Gamma \times \Gamma $ on $ \ell^\infty(\Gamma)/I $ is amenable \cite[Prop. 15.2.3]{brown2008textrm}.

The following follows immediately from the definition of almost malnormality (see, for instance, Lemma 3.3 of \cite{DKE24properproximalgroups} Lemma 3.1 of \cite{kalantar2026conjugacycoamenability}).
\begin{lemma}\label{lem: sHt orthogonal}
    If $H\subset G$ is an almost malnormal subgroup, then the family of characteristic functions $ \{1_{sHt}\}_{[s],[t^{-1}]\in \Gamma/H} $ has mutually disjoint support in $ c_0(\Gamma,\{H\})/c_0(\Gamma) $.
    
    In particular, we have the decomposition of spectrum $$\text{Sp}(c_0(\Gamma,\{H\})/c_0(\Gamma)) = \bigsqcup_{[s],[t^{-1}]\in \Gamma/H} \text{Sp}(\ell^\infty(sHt )/c_0) = \bigsqcup_{[s],[t^{-1}]\in \Gamma/H} s(\beta H \backslash H)t, $$
    where $ \beta H := \text{Sp}(\ell^\infty(H)) $ is the Stone-\v{C}ech compactification of $H$.
\end{lemma}

\begin{lemma}
    If each $ H_i $ is almost malnormal and biexact, then the action $ \Gamma\times \Gamma $ on $ c_0(\Gamma,\{H_i\}_{i\in I})/c_0( \Gamma) $ is amenable.
\end{lemma}
\begin{proof}
    We first assume that $\{H_i\}_{i\in I}$ consists of only one element $H$. It suffices to show that there exists a net \[m_i:  \text{Sp}(c_0(\Gamma,\{H\})/c_0(\Gamma)) = \bigsqcup_{[s],[t^{-1}]\in \Gamma/H} s(\beta H\backslash H)t \to \text{Prob}(\Gamma\times \Gamma)\]
    such that for any compact subset $ K\subset \text{Sp}(c_0(\Gamma,\{H\})/c_0(\Gamma)) $ and $ s,t\in \Gamma $,
    \[ \lim_{i}\sup_{x\in K}\|(s,t^{-1})\cdot m_i(x)-m_i( sxt ) \|_1 = 0. \]
    Since $\text{Sp}(c_0(\Gamma,\{H\})/c_0(\Gamma)) = \bigsqcup_{[s],[t^{-1}]\in \Gamma \backslash H} s(\beta H/H)t $ and each $ s(\beta H\backslash H)t $ is compact, fixing coset representatives $ \rho: \Gamma/H \to \Gamma $, we may assume $K$ is a finite union $  K=\bigsqcup_{k=1}^n s_k (\beta H \backslash H)t_k, $ with $s_k,t_k^{-1}\in \rho(\Gamma/H)$.
    As $H$ is biexact, the action $H\times H$ on $ \ell^\infty(H)/c_0(H) $ is amenable. Therefore, there exists a net $ \bar{m}_i: \text{Sp}(\ell^\infty(H)/c_0(H)) = \beta H \backslash H \to \text{Prob}(H\times H)
    $ such that for all $s,t\in H$,
    $$ \lim_i \sup_{x\in \beta H\backslash H} \|(s,t^{-1})\cdot \bar{m}_i(x)-\bar{m}_i(sxt)\|_1=0. $$
    For each $x\in \beta H\backslash H $ and $a,b\in \Gamma/H$, we define $m_i(\rho(a)x\rho(b)^{-1})\in \text{Prob}(\Gamma\times \Gamma)$ with support in $ \rho(a)H\times \rho(b)H $ as
    \[  m_i(\rho(a)x\rho(b)^{-1} ) :=(\rho(a),\rho(b))\cdot \bar{m}_i(x). \]
    Now, we have for $s,t\in \Gamma$,
    \begin{align*}
        &\sup_{y\in K}\|(s,t^{-1})\cdot m_i(y)-m_i(syt)\|_1\\ =& \sup_{x\in \beta H\backslash H}\sup_{ 1\leq k\leq n}\|(s,t^{-1})\cdot m_i( s_kx t_k )-m_i(ss_kxt_kt)  \|_1\\
        =& \sup_{x\in \beta H\backslash H} \sup_{1\leq k\leq n}\|(ss_k,(t_k t)^{-1})\cdot \bar{m}_i(x)\\ &- (\rho(ss_k),\rho((t_k t)^{-1})^{-1})\cdot \bar{m}_i \left( \rho(ss_k)^{-1}ss_kx tt_k \rho( (t_k t)^{-1} )\right)  \|_1\\
        =& \sup_{x\in \beta H\backslash H} \sup_{1\leq k\leq n}\|\left( \rho(ss_k)^{-1}ss_k,(t_k t)^{-1}\rho( (t_k t)^{-1} )\right)\cdot \bar{m}_i(x)\\ &- \bar{m}_i\left( \rho(ss_k)^{-1}ss_kx tt_k \rho( (t_k t)^{-1} )\right)  \|_1,
    \end{align*}
    which converges to $0$ by the definition of $ \bar{m}_i $.

    The case when the family consists of two elements $ \{H_1,H_2\} $ follows from the $ (\Gamma\times\Gamma)$-equivariant short exact sequence
    \[ 0\to c_0(\Gamma,\{H_2\})/c_0(\Gamma,\{H_1\cap H_2\}) \to c_0(\Gamma,\{H_1,H_2\})/c_0(\Gamma) \to c_0(\Gamma,\{H_1\})/c_0(\Gamma) \to 0. \]
    The same argument shows that the $(\Gamma\times \Gamma)$-action on $ c_0(\Gamma,\{H_i\}_{i\in F})/c_0(\Gamma) $ is amenable for any finite subsets $F\subset I$. Therefore, $ \left(c_0(\Gamma,\{H_i\}_{i\in I})/c_0(\Gamma)\right)\rtimes_r (\Gamma \times \Gamma) $ is nuclear since it is the inductive limit of $ \left(c_0(\Gamma,\{H_i\}_{i\in F})/c_0(\Gamma)\right)\rtimes_r (\Gamma \times \Gamma) $.
\end{proof}

\begin{corollary}\label{cor: group case main thm}
    If $\Gamma$ is biexact relative to $ \{H_i\}_{i\in I} $, and each $H_i$ is biexact and almost malnormal, then $\Gamma$ is biexact.
\end{corollary}
\begin{proof}
    This follows from the previous Lemma and the $ (\Gamma \times \Gamma) $-equivariant short exact sequence
    \[ 0\to c_0(\Gamma,\{H_i\}_{i\in I})/c_0(\Gamma)\to\ell^\infty(\Gamma)/c_0(\Gamma) \to \ell^\infty(\Gamma)/c_0(\Gamma,\{H_i\}_{i\in I})\to 0. \]
\end{proof}

\begin{remark}
    The main difficulty for generalizing this proof for von Neumann algebras is that we do not have a von Neumann algebra analogue of the biexactness characterization using the $\Gamma \times \Gamma$ action.
\end{remark}

\subsection{A proof via the bidual}

% The following characterization of relative biexactness was suggested to us by J. Peterson. In particular, using this, we could have another proof of Corollary \ref{cor: group case main thm} for countable $\Gamma$ in the same way as in the proof of our main theorem.

The following characterization of relative biexactness, suggested to us by J. Peterson, serves as a natural analog to Theorem \ref{thm: bidual characterization of biexact}. We include it here because it is of independent interest and conceptually bridges our general von Neumann algebra approach with the group setting. Specifically, while we have already provided a short proof of Corollary \ref{cor: group case main thm}, this characterization provides the necessary tool to adapt the techniques from our main theorem into an alternative, purely group-theoretic proof for countable $\Gamma$. 

\begin{thm}\label{thm: Jesse}
    Let $\Gamma$ be a countable discrete exact group, and $I\subset \ell^\infty(\Gamma)$ be a $\Gamma$-boundary piece. The following are equivalent
    \begin{enumerate}
        \item $\Gamma$ is biexact relative to $I$.
        \item For every finite subset $ E\subset \Gamma $ and $\varepsilon>0$, there exists $\mu: \Gamma\to \text{Prob}(\Gamma)$ such that for all $s,t\in \Gamma$, the subset
        \[ \{x\in \Gamma: \|\mu(sxt)-s\cdot\mu(x)\|_1\geq \varepsilon\} \] is small relative to $I$. 
        \item The left action of $\Gamma$ on the abelian von Neumann algebra
    \[ \left [\left(\ell^\infty( \Gamma )/I\right)^{**}\right]^{\Gamma_r} \] is amenable in the sense of Zimmer, where $ \left [\left(\ell^\infty( \Gamma )/I\right)^{**}\right]^{\Gamma_r}$ means the right $\Gamma$-invariant part of $ \left(\ell^\infty( \Gamma )/I\right)^{**}  $.
    \item For every finite subset $E\subset \Gamma$ and every $\delta>0$, there exist a finite subset $K\subset \Gamma$ and elements $a_g\in (\ell^\infty(\Gamma)/I)_+$ with $g\in K$, such that $ \sum_{g\in K}a_g=1 $ and, putting $a_g=0$ outside $K$, for all $s\in E$,
    \[ \|\sum_{g\in \Gamma}|s\cdot a_g-a_{sg}|\|<\delta,\quad \|   \sum_{g\in \Gamma}| \rho_s(a_g)-a_{g}|\|<\delta\]
    \end{enumerate}
\end{thm}
\begin{proof}
    The equivalence between (1) and (2) follows from \cite[Remark 15.1.3]{brown2008textrm}.
    
    (1) $\implies$ (3): Suppose that $\Gamma$ is biexct relative to $I$, we have that $ \S_I(\Gamma)$ is $\Gamma$-amenable. Since $ \iota:\S_I(\Gamma)\to \left [\left(\ell^\infty( \Gamma )/I\right)^{**}\right]^{\Gamma_r} $ is $\Gamma$-equivariant u.c.p, this implies that $ \left [\left(\ell^\infty( \Gamma )/I\right)^{**}\right]^{\Gamma_r}  $ is $\Gamma$-amenable.

    (3) $\implies $ (4): Fix a finite subset $ E\subset \Gamma $ and $\delta>0$. Suppose that no such $K\subset \Gamma$ and tuple $(a_g)_{g\in K}$ satisfy (4). Consider the convex set $\mathcal{C}$ of tuples $ (x_s,y_t)_{s,t\in E}\in \bigoplus_{s\in E}(\ell^\infty(\Gamma)/I)_+ \oplus \bigoplus_{s\in E}(\ell^\infty(\Gamma)/I)_+ $ such that there exists a finite partition of unity $(a_g)_{g\in \Gamma}\subset (\ell^\infty(\Gamma)/I)_+$ with $ \sum a_g=1 $, and
    \[ x_s\geq \sum_g |s\cdot a_g-a_{sg}|,\quad y_t\geq \sum_{g} |\rho_t(a_g)-a_g|. \]

    Since we are assuming that (4) does not hold, $\mathcal{C}\cap \{(x_s,y_t)_{s,t\in E}: \|x_s\|,\|y_t\|<\delta, \forall s,t\in E\} = \varnothing$. By Hahn-Banach separation theorem, there exist nonzero positive linear functionals, $ \varphi_s,\psi_t\in (\ell^\infty(\Gamma)/I)_+^* $ such that $ \sum_{s\in E}\|\varphi_s\|+\sum_{t\in E}\|\psi_t\|=1 $ and for any finite partition of unity $ (a_g)_{g\in \Gamma}\subset (\ell^\infty(\Gamma)/I)_+ $,
    \begin{equation}\label{eq: section 5}
        \sum_{s\in E} \varphi_s \left(\sum_g |s\cdot a_g-a_{sg}|\right)+ \sum_{t\in E}\psi_t\left( \sum_{g} |\rho_t(a_g)-a_g|\right)\geq \delta.
    \end{equation} 
    We now contradict this with the Zimmer amenability.
    
    Since $ N:=\left [\left(\ell^\infty( \Gamma )/I\right)^{**}\right]^{\Gamma_r}  $ is $\Gamma$-amenable in the sense of Zimmer, since $\Gamma$ is exact, the action of $\Gamma$ on the spectrum $ \text{Sp}(N)$ is amenable. We can choose a map $m:\text{Sp}(N)\to \text{Prob}(\Gamma) $ such that $ \sup_x\| s\cdot m(x)-m(sx) \|<\delta/4 $ for all $E\subset \Gamma$, we may further assume that $\text{supp}(m(x))$ is contained in a finite subset $K\subset \Gamma$ for all $x\in \text{Sp}(N)$ (see e.g. \cite[Lemma 4.3.8]{brown2008textrm}). Define for each $g\in K$, define $\alpha_g\in N_+$ as $\alpha_g(x):=m(x)(g) $ for $x\in \text{Sp}(N)$, and set $\alpha_g=0$ for $g\in \Gamma/K$, then we can check from the definition that for all $s\in E$, $\sum_{g\in \Gamma} \alpha_g=1$ and \[ \|\sum_{g\in \Gamma}|s\cdot\alpha_g -\alpha_{sg}|\|<\delta/4. \]
    We now contradict \eqref{eq: section 5} using approximation of $(\alpha_g)_g$. Let $S$ be the finite subset of linear functionals on $ \ell^\infty(\Gamma)/I $,
    \[ S:= \{ \varphi_s,\psi_t,\varphi_{s}\circ \lambda_g,\varphi_s\circ \rho_g,\psi_{t}\circ \lambda_g,\psi_t\circ \rho_g: s,t,g\in E \}. \]
    Choose a finite partition of unity $(a_g)_{g\in K}$ close to $ (\alpha_g)_{g\in K} $ in the sense that
    \[ \sum_{\varphi\in S}\varphi( \sum_{g\in K}|a_g-\alpha_g| )<\delta/4. \]
    Therefore, we obtain for all $s\in E$
    \begin{align*}
        \varphi_s\left(\sum_{g\in \Gamma}|s\cdot a_g -a_{sg}| \right) < \varphi_s\left(\sum_{g\in \Gamma}|s\cdot\alpha_g -\alpha_{sg}|\right)+\delta/4<\delta/2.
    \end{align*}
    Similarly, for all $t\in E$
    \begin{align*}
        \psi_t\left(\sum_{g\in \Gamma}|\rho_t( a_g) -a_{g}| \right) < \psi_t\left(\sum_{g\in \Gamma}|\rho_t(\alpha_g) -\alpha_{g}|\right)+\delta/4=\delta/4.
    \end{align*}
    Altogether, this contradicts \eqref{eq: section 5}.
    
    (4) $\implies$ (2): Let $(a_g)_{g\in K}\subset (\ell^\infty(\Gamma)/I)_+$ be the elements given by (4), choose their positive lifts $f_g\in \ell^\infty(\Gamma)$ such that $\sum_{g\in K} f_g=1 $. Define $\mu: \Gamma\to \text{Prob}(\Gamma)$, \[ \mu(x)(g):=f_g(x). \] 
    For $s,t\in E$, we set $ H_{s,t}(x) := \|\mu(sxt)-s\cdot \mu(x)\|_1  $ for $x\in \Gamma$, then
    \[ H_{s,t} = \sum_{g\in K} |s^{-1}\cdot \rho_t(f_g)- f_{s^{-1}g}| \leq \sum_{g\in \Gamma}|s^{-1}\cdot \rho_t(f_g)-s^{-1}\cdot f_g|+ \sum_{g\in \Gamma}|s^{-1}\cdot f_g-f_{s^{-1}g}|.\]
    Let $Q:\ell^\infty(\Gamma)\to \ell^\infty(\Gamma)/I$ be the quotient map, then the above equation implies
    \[ Q(H_{s,t})\leq \sum_{g\in \Gamma}|\rho_t(a_g)-a_g|+ \sum_{g\in \Gamma}|s^{-1}\cdot a_g-a_{s^{-1}g}| \leq 2\delta.
\]
Therefore, $ \{x\in \Gamma: H_{s,t}\geq 2\delta\} $ is small relative to $I$.
\end{proof}

We now provide a sketch of the second proof of Corollary \ref{cor: group case main thm} using (3) of this characterization. For simplicity, we will only deal with the case with $ I=c_0(\Gamma,\{H\})$ for a single almost malnormal subgroup $H$. 

Suppose $ \Gamma$ is biexact relative to $H $ and $H$ is almost malnormal and biexact. Since \[ \left [\left(\ell^\infty( \Gamma )/c_0(\Gamma) \right)^{**}\right]^{\Gamma_r} =  \left [\left(\ell^\infty( \Gamma )/c_0(\Gamma,\{H\}) \right)^{**}\right]^{\Gamma_r}\oplus \left[ (c_0(\Gamma,\{H\})/c_0(\Gamma))^{**}\right]^{\Gamma_r},\]
by (3) of Theorem \ref{thm: Jesse}, it suffices to show that the left action of $\Gamma$ on $\left[ (c_0(\Gamma,\{H\})/c_0(\Gamma))^{**}\right]^{\Gamma_r}$ is amenable.

The following is an analogue of our key Lemma \ref{key lemma}.
\begin{lemma}
    Let $ C\subset \ell^\infty(\Gamma) $ be the subalgebra of functions $f$ such that $\forall\varepsilon>0$, $\{g: |f(g)|\geq \varepsilon \}$ is contained in a finite union of left cosets $ \bigcup_{i} g_iH $. Then there exists a left $\Gamma$-invariant c.c.p. map
    \[ \phi_0:C\to \left[ (c_0(\Gamma,\{H\})/c_0(\Gamma))^{**}\right]^{\Gamma_r}, \]
    which extends to a u.c.p. map
    \[  \phi:C^{**}\to \left[ (c_0(\Gamma,\{H\})/c_0(\Gamma))^{**}\right]^{\Gamma_r}\]
\end{lemma}
\begin{proof}
    Since $H$ is biexact, there is a left $H$-invariant net of u.c.p. maps $ E_H^{(n)}:\ell^\infty(H)\to \S(H)^{**} $.

    Fix again left cosets representatives $ \rho:  \Gamma/H \to \Gamma $, we define the u.c.p. maps
    $ \tilde{E}_H: \ell^\infty(H)\to \ell^\infty( \Gamma )^{**} $,
    \[ \tilde{E}_H(f)=\sum_{ y\in \rho(\Gamma/H)}\rho_{y^{-1}}(E_H(f)),  \]
    One can check that the image of $\tilde{E}_H$ is contained in
    \[ \{f\in \ell^\infty(\Gamma)^{**}: f-\rho_t(f)\in B^{**},\forall t\in \Gamma \} \]
    where $B \subset \ell^\infty(\Gamma)$ is the subalgebra of functions $f$ such that for $\varepsilon >0$, $ \{ g:|f(g)|\geq \varepsilon \} $ is contained in a finite union of right cosets of $H$.

    We can now define the desired c.c.p. map $\phi_0: C\to \left[ (c_0(\Gamma,\{H\})/c_0(\Gamma))^{**}\right]^{\Gamma_r}$ as
    \[ \phi_0(y\cdot f):= y\cdot\tilde{E}_H(f)p,\forall y\in \rho(\Gamma/H),f\in \ell^\infty(H), \]
    where $p = \left(\vee_{g\in \Gamma}\iota(1_{Hg})\right)-\vee_{g\in \Gamma}\iota(1_{\{g\}}) = \left(\vee_{g\in \Gamma}\iota(1_{Hg})\right)-1_{c_0(\Gamma)^{**}}$. Indeed, to see that $\phi_0$ is well-defined, we note that $\phi_0$ is well-defined for functions supported on a finite union of left cosets. Note also that the support of $y\cdot \tilde{E}_H(f)p $ is majorized by $ \vee_{g\in \Gamma} \iota(1_{yHg})-1_{c_0(\Gamma)^{**}} $, which are mutually orthogonal for distinct $y\in \rho(\Gamma/H)$ by Lemma \ref{lem: sHt orthogonal}. Therefore, $\phi_0$ is contactive and thus extends to the whole $C$.

    Finally, to see that $\phi$ is unital, as $ q:=\vee_{g\in \Gamma} \iota(1_{gH}) \in (\ell^\infty(\Gamma))^{**}$ is the identity of $C^{**}$, we simply compute $  \phi(\vee_{g\in \Gamma}\iota(1_{gH}))= \vee_{ g,h\in \Gamma} \iota(1_{gHh})-1_{c_0(\Gamma)^{**}}. $
\end{proof}

To finish the second proof, let again $ q:=\vee_{g\in \Gamma} \iota(1_{gH}) \in (\ell^\infty(\Gamma))^{**}$ be the identity of $ C^{**}$ inside $(\ell^\infty(\Gamma) )^{**}$, then we have a $\Gamma$-invariant u.c.p. map \[ \phi\circ\text{Ad}(q)\circ\iota: \ell^\infty(\Gamma) \to C^{**} \to \left[ (c_0(\Gamma,\{H\})/c_0(\Gamma))^{**}\right]^{\Gamma_r}, \]
which implies the amenability of $ \left[ (c_0(\Gamma,\{H\})/c_0(\Gamma))^{**}\right]^{\Gamma_r} $ by the exactness of $\Gamma$.

\bibliographystyle{alpha}
\bibliography{main}
\end{document}